\documentclass[11pt]{amsart}

\usepackage{setspace}

\usepackage{amsmath,amssymb,amsthm,amscd,yhmath,mathtools,comment,mathrsfs,bbm,hyperref,thmtools}
\usepackage{fancyhdr}
\usepackage{cancel}
\usepackage[pdftex]{graphicx}
\graphicspath{ {./images/} }
\usepackage{multicol}
\usepackage{textpos}
\usepackage{geometry}
\usepackage{stmaryrd}

\usepackage{enumitem}

\usepackage{array}
\usepackage{footnote}
\makesavenoteenv{tabular}

\newcolumntype{o}{@{}>{{}}c<{{}}@{}}

\newenvironment{MyLinSys*}[1]
{\begin{array}{*{#1}{ro}r}}{\end{array}}

\title{Codimension of Brill-Noether Loci on Abelian Surfaces and the First Segre Invariant}
\author{Jesse Chien}
\date{}

\usepackage{tikz,pgfplots}
\usepackage{tikz-cd}
\usetikzlibrary{decorations.markings}
\usetikzlibrary{arrows}
\usetikzlibrary{positioning}
\usetikzlibrary{calc}

\theoremstyle{plain}
\newtheorem{theorem}{Theorem}[section]

\newtheorem{proposition}[theorem]{Proposition}
\newtheorem{lemma}[theorem]{Lemma}

\theoremstyle{definition}
\newtheorem{defn}[theorem]{Definition}

\newtheorem{remark}[theorem]{Remark}
\newtheorem{example}[theorem]{Example}
\newtheorem{notation}[theorem]{Notation}
\newtheorem{assumptions}[theorem]{Assumptions}

\newtheorem*{defn*}{Definition}
\newtheorem*{claim*}{Claim}
\newtheorem*{strategy*}{Strategy}
\newtheorem*{notation*}{Notation}
\newtheorem*{setup*}{Setup}
\newtheorem*{theorem*}{Theorem}
\newtheorem*{initial reduction*}{Initial Reduction}
\newtheorem*{organization*}{Organization}

\newcommand{\CC}{\mathbb{C}}

\newcommand{\RR}{\mathbb{R}}
\newcommand{\ZZ}{\mathbb{Z}}
\newcommand{\PP}{\mathbb{P}}

\newcommand{\Oc}{\mathcal{O}}

\newcommand{\imply}{\Rightarrow}
\newcommand{\im}{\operatorname{im}}

\newcommand{\Hom}{\operatorname{Hom}}
\newcommand{\Ext}{\operatorname{Ext}}
\newcommand{\ext}{\operatorname{ext}}
\newcommand{\rank}{\text{rank}}

\begin{document}
\maketitle
\begin{abstract}
Let $X$ be an abelian surface of Picard rank 1 with ample generator $H$. We show that, for a Mukai vector $\mathbf{v}=(r,mH,\chi)$ with $r,m\geq 1$ and $\chi\leq 0$, the first Brill-Noether loci $\{E \;|\; H^0(E)\neq 0\}$ in $M_X(\mathbf{v})$ is nonempty of the expected codimension $1-\chi$. As an application, we study the rank 1 Segre invariant on these abelian surfaces. Specifically, for $\mathbf{v}$ a Mukai vector of a $\mu$-stable sheaf of $\mathrm{rank}\geq 2$, we compute the Segre invariant of the general sheaf in $M_X(\mathbf{v})$, and we determine the numbers which arise as the Segre invariant of a sheaf in $M_X(\mathbf{v})$.
\end{abstract}

\section{Introduction}
For a smooth projective surface $X$ and a fixed Mukai vector $\mathbf{v}$, there are moduli spaces $M_X^s(\mathbf{v})$ of Gieseker stable sheaves on $X$ with Mukai vector $\mathbf{v}$. In these spaces, there are special subschemes 
\[
B_X^k = \{E\;|\; h^0(E)\geq k\}\subset M_X^s(\mathbf{v})
\]
known as Brill-Noether loci. The geometry of these subschemes, generalizing Brill-Noether theory on curves, has been extensively studied on many classes of surfaces (for rational surfaces, see \cite{ch18}, \cite{ch20}, \cite{gh94}, \cite{gll22}, \cite{gl26}; for abelian surface, see \cite{coskunnueryoshioka2}, \cite{bayerli}; for K3 surfaces, see \cite{cny23}, \cite{l06}, \cite{l12}). For a survey of the field, we refer the reader to \cite{coskunnueryoshioka}\par

In this paper, we will study the first Brill-Noether loci of moduli spaces of sheaves on abelian surfaces. The geometry of these spaces in general have also been widely studied, for example in the seminal work of Mukai and Yoshioka (see, for example, \cite{mukai81}, \cite{mukai87}, \cite{myy1}, \cite{myy2}, \cite{yoshioka}, \cite{yoshioka03}, \cite{yoshioka09b}). 

 \begin{defn}
 Let $(X,H)$ be a polarized abelian surface of Picard rank 1, with $H$ the ample generator, and let $M_X^s(\mathbf{v})$ denote the moduli space of stable sheaves of Mukai vector $\mathbf{v}$ on $X$ (with respect to $H$). For this paper, the \emph{Brill-Noether locus} of $M_X^s(\mathbf{v})$ will denote the closed subscheme
 \[
 B^1_X(\mathbf{v}) = \{E\;|\; H^0(E)\neq 0\}\subset M_X^s(\mathbf{v}).
 \]
  The analogous definition for the Brill-Noether loci $\mathcal{B}^1_X(\mathbf{v})$ can be made for the stack $\mathcal{M}_X(\mathbf{v})$ of Gieseker semistable sheaves of Mukai vector $\mathbf{v}$.
 \end{defn}

 In this paper, we prove the following theorem.

\begin{theorem}[Theorem \ref{nonemptiness} and Theorem \ref{BNmain}]\label{BNintro}
Let $(X,H)$ be a polarized abelian surface of Picard rank 1, with $H$ the ample generator. Let $\mathbf{v}=(r,mH,\chi)$ be the Mukai vector of a stable sheaf with $r,m\geq 1$ and $\chi\leq 0$. Then the Brill-Noether locus $\mathcal{B}^1_X(\mathbf{v})$ in the stack $\mathcal{M}_X(\mathbf{v})$ is nonempty and has the expected codimension of $1-\chi$. In particular, the Brill-Noether locus $B^1_X(\mathbf{v})$ in the coarse space $M_X^s(\mathbf{v})$ is nonempty and has the expected codimension $1-\chi$. 
\end{theorem}

\begin{remark}
The case $m=1$ is contained in Theorem 1.2 of \cite{bayerli}. The analogous question for K3 surfaces is answered in \cite{l12}.
\end{remark}

\begin{remark}
Recall that the Brill-Noether locus has a determinental structure, so it's codimension is $\leq 1-\chi$ (see \cite{roa}). Thus, it suffices to prove that the codimension is $\geq 1-\chi$. 
\end{remark}

As an application, we will study the rank 1 Segre invariant for stable sheaves on abelian surfaces. We first recall the definition.

\begin{defn}
Let $(X,H)$ be a polarized smooth projective surface. Let $E$ be a coherent sheaf of rank $r$, and let $1\leq k\leq r-1$. The \emph{$k$th Segre invariant} of $E$ is defined to be
\[
s_k(E)=\min_{F\subset E}\big\{ (kc_1(E)-rc_1(F)).H\},
\]
where the minimum is taken over all subsheaves $F\subset E$ of rank $k$. 
\end{defn}
\begin{remark}
One should think of $s_k(E)$ as being the maximum degree of a rank $k$ subsheaf of $E$. The definition is set up so that $E$ is $\mu$-semistable (resp. stable) if and only if $s_k(E)\geq 0$ (resp. $>0$) for all $1\leq k\leq r-1$. \par
Note also that $s_k$ is a lower semicontinuous function and thus induces a stratification on $M_X(\mathbf{v})$.
\end{remark}

The Segre invariant on curves is well understood (see, for example, \cite{brambilapazlange}, \cite{costamiroroig}, \cite{langenarasimhan}, and \cite{russoteixidor}). On surfaces, much less is known, though there has been some very recent progress studying $s_1$ on rational surfaces (see \cite{ri26}, \cite{ro26}, \cite{roa}, and \cite{gl26}). In particular, the $\PP^2$ case has been fully worked out in \cite{gl26}.

Here, we will study the first Segre invariant $s_1$ on an abelian surface $X$ of Picard rank 1.  Our main result is the following.

\begin{theorem}[Theorem \ref{genericsegre} and Theorem \ref{segrenonemptiness}]\label{segreintro}
Let $(X,H)$ be an abelian surface of Picard rank 1, with $H$ the ample generator, and let $\mathbf{v}=(r,mH,\chi)$ be a Mukai vector with $r\geq 2$ and $\mathbf{v}^2>0$, and $\mathbf{v}\neq (r,0,-1)e^{tH}$. 
\begin{enumerate}
\item Let $k_{\mathbf{v}}$ denote the largest integer $<\mu(\mathbf{v})$ for which $\chi(L,E)\geq -1$, where $\mathbf{v}(E)=\mathbf{v}$ and $L$ is a degree $k_{\mathbf{v}}$ line bundle. Then the first Segre invariant for the general $E\in M_X(\mathbf{v})$ is given by 
\[
(m-rk_{\mathbf{v}})H^2.
\]
 Equivalently, for the general $E$, the maximal degree of a rank 1 subsheaf is $k_{\mathbf{v}}$.
\item For every $k$ for which $k_{\mathbf{v}}\leq k\leq \lceil \mu(\mathbf{v})\rceil-1$, $(m-rk)H^2$ is the Segre invariant of a $\mu$-stable sheaf in $M_X(\mathbf{v})$. 
\end{enumerate}
\end{theorem}  
\begin{remark}
The condition on $\mathbf{v}$ just ensures that there are $\mu$-stable locally free sheaves in $M_X(\mathbf{v})$. For more details, see Remark \ref{assumelocallyfree}.
\end{remark}

\subsection*{Strategy}
The main technique we will use is wall-crossing and Fourier-Mukai transforms. First introduced by Bridgeland (see \cite{br07},\cite{bridgelandk3},\cite{br09}), Bridgeland stability conditions and wall-crossing have seen large success in dealing with questions regarding the geometry of moduli of sheaves (see, for example, \cite{bm14a},\cite{bayermacri},\cite{bayerli}). The idea is to start with $t_0\gg 0$ so that $\mathcal{M}_{\sigma_{(0,t_0)}}(\mathbf{v}) \cong \mathcal{M}_X(\mathbf{v})$. We then travel along the path $\sigma_{(0,t)}$, showing that, at each wall, we don't lose too many objects in the Brill-Noether loci. Once we reach $\mathcal{M}_{\sigma_{(0,t_1)}}(\mathbf{v})$ for $t_1\ll 1$, we apply a Fourier-Mukai transform that reinterprets the question into one about singular sheaves on the dual abelian surface.

\begin{remark}
As we will see, to bound the lost objects, we will need an induction argument on $c_1$. The strategy of bounding objects lost at each wall, as well as the strategy of inducting on $c_1$, is very similar to the methods used in \cite{gl26}.
\end{remark}

\subsection*{Organization}
In section \ref{prelim}, we recall the necessary background. In section \ref{nonempty}, we use elementary modifications to prove that the Brill-Noether loci are nonempty. In section \ref{codim}, we use wall-crossing and Fourier-Mukai transforms to prove that the Brill-Noether loci have the expected codimension, thus proving Theorem \ref{BNintro}. In section 5, we will study the Segre invariant and prove Theorem \ref{segreintro}.

\subsection*{Acknowledgements}
The author would like to thank his advisor, \.{I}zzet Co\c{s}kun, for his continued patience and for the many invaluable conversations and suggestions throughout the preparation of this work.\par
The author would also like to thank Thomas Goller and Yinbang Lin for very valuable correspondence and for pointing out notable similarities between our methods and their methods in \cite{gl26}.  

\subsection*{AI Disclosure}
This project began as a study of the Segre invariant. The author realized that, to compute the Segre invariant of the generic sheaf, one would need to know that, if $\mathbf{v}$ has $\chi\leq -2$, then the Brill-Noether locus of $M_X(\mathbf{v})$ has codimension $\geq 3$ (see the proof of Theorem \ref{genericsegre}). ChatGPT suggested a strategy for doing this, which the author realized could be extended to prove Theorem \ref{BNmain}. The author then worked with ChatGPT to produce the proof in its current form. In particular, many of the wall-crossing estimates were generated by ChatGPT, and the author carefully checked and simplified these calculations. The author then realized that the proof of Theorem \ref{BNmain} can be applied to study the Segre invariant (see the proof of Theorem \ref{genericsegre}). This and all other results in Section \ref{segresection} were done independently of AI. Finally, ChatGPT was also used to proofread the text.

\section{Preliminaries}\label{prelim}
In this section, we recall the necessary background and results on moduli of sheaves on surfaces on abelian varieties. Much of this section is based on the background section of \cite{coskunnueryoshioka2}. 

\subsection{Mukai Lattice}
Let $X$ be an abelian surface. The even singular cohomology of $X$ 
\[
H^{2*}(X,\ZZ)=\ZZ\oplus H^2(X,\ZZ)\oplus \ZZ
\]
has a natural pairing called the \emph{Mukai pairing}. If $x=\langle x_1,x_2,x_3\rangle$ and $y=\langle y_1,y_2,y_3\rangle$, then 
\[
\langle x,y\rangle =x_2\cdot y_2-x_1y_3-x_3y_1,
\]
where $\cdot$ denotes the cup product. If $E$ is a coherent sheaf on $X$, the \emph{Mukai vector} of $E$ is 
\[
\mathbf{v}(E)=\mathrm{ch}(E).
\]
For coherent sheaves $E$ and $F$, define the \emph{Euler characteristic of the pair} $E$ and $F$ to be
\[
\chi(E,F)=\sum_{i=0}^2(-1)^i\Ext^i(E,F).
\]
The Hirzebruch-Riemann-Roch formula yields the following simple relation:;
\[
\langle \mathbf{v}(E), \mathbf{v}(F)\rangle =-\chi(E,F).
\]

\subsection{Moduli of Sheaves on Surfaces}
Here, we introduce the basics of the moduli of sheaves on Picard rank 1 abelian surfaces. For an introduction to the general theory, we refer the reader to \cite{huybrechtslehn}.\par 
Let $(X,H)$ be a polarized abelian surface of Picard rank 1, where $H$ is the ample generator and $H^2=2d$. Let $E$ be a rank r sheaf of pure-dimension $d$ on $X$ with Mukai vector $\mathbf{v}=(r,mH,\chi)$. The \emph{Hilbert polynomial} $P_E(m)$ and the \emph{reduced Hilbert polynomial} $p_E(m)$ of $E$ (with respect to $H$) are 
\begin{align*}
P_E(m)&=\chi(E(m))=a_d\frac{m^d}{d!}+\text{l.o.t}.,\\
p_E(m)&=\frac{1}{a_d}P_E(m).
\end{align*}
In the case $E$ is torsion-free (i.e. $\dim(E)=2$), we have the following explicit formulas by Riemann-Roch.
\begin{align*}
P_E(m)&=\frac{2dr}{2}k^2+2dmk+\chi \\\
p_E(m)&=k^2+\frac{m}{r}k+\frac{\chi}{2dr}.
\end{align*}
A sheaf $E$ is Gieseker stable (resp. semistable) if for every proper subsheaf $F\subset E$, we have 
\[
p_F(m)<p_E(m)\hspace{0.3cm}\text{(resp. $p_F(m)\leq p_E(m)$)}
\]
for $m\gg0$. The \emph{slope} $\mu(E)=\mu_H(E)$ of $E$ is defined as 
\[
\mu_H(E)=\frac{m}{r}.
\]
A sheaf $E$ is $\mu$-stable (resp. $\mu$-semistable) if, for every proper subsheaf $F$, 
\[
\mu(F)<\mu(E)\hspace{0.3cm}\text{(resp. $\mu(F)\leq \mu(E)$)}
\]
For torsion-free sheaves, we have
\[
\text{$\mu$-stable}\imply \text{Gieseker stable}\imply \text{Gieseker semistable}\imply \text{$\mu$-semistable}
\]
Given $(X,H)$ and a Mukai vector $\mathbf{v}$, there are projective moduli spaces $M_{X,H}(\mathbf{v})$ parametrizing S-equivalence classes of Gieseker semistable sheaves with Mukai vector $\mathbf{v}$ (\cite{gieseker},\cite{maruyama}). We denote the locus of stable sheaves as $M_{X,H}(\mathbf{v})^s$. \par
Also denote by $\mathcal{M}_{X,H}(\mathbf{v})$ and $\mathcal{M}^s_{X,H}(\mathbf{v})$, respectively, the stack of Gieseker semistable sheaves of Mukai vector $\mathbf{v}$ and the stack of Gieseker stable sheaves of Mukai vector $\mathbf{v}$. When $X$ and $H$ are understood, we may drop them from the notation.\par
The main existence theorem for stable sheaves we will use are the following results. 

\begin{theorem}\label{sheavesexistence}
Let $(X,H)$ be a polarized abelian surface of Picard rank 1 . Let $\mathbf{v}$ be a Mukai vector with $\mathbf{v}^2\geq 0$. Assume $\mathbf{v}=k\mathbf{v}_0$ for a primitive Mukai vector $\mathbf{v}_0$. 
\begin{enumerate}
\item (\cite{myy1}, Lemma 5.3.2) If $\mathbf{v}^2>0$, then 
\[
\dim \mathcal{M}_{X,H}(\mathbf{v}) = \mathbf{v}^2+1.
\]
 In particular, $\dim M_{X,H}(\mathbf{v})=\mathbf{v}^2+2$. If $\mathbf{v}^2=0$, then 
 \[
 \dim\mathcal{M}_{X,H}(\mathbf{v}) = k.
 \] 
\item \cite{kuriharayoshioka} Assume $r\geq 2$, $\mathbf{v}^2>0$, and $\mathbf{v}\neq (r,0,-1)e^\eta$ for $\eta\in \mathrm{NS}(X)$. Then there is a $\mu_H$-stable vector bundle with $\mathbf{v}(E)=\mathbf{v}$. 
\end{enumerate}
\end{theorem}

\begin{remark}\label{stablesheavesexist}
In this paper, we will mainly be dealing with Mukai vectors with $\chi\leq 0$. In this case, one can check that the $\mathbf{v}\neq (r,0,-1)e^\eta$ assumption in part (2) of the theorem always applies, i.e. if $r\geq 2,\mathbf{v}^2>0$, then $M_{X,H}(\mathbf{v})$ always contains $\mu$-stable sheaves. 
\end{remark}

\subsection{Bridgeland stability on abelian surfaces}
Here, we recall the basics of Bridgeland stability conditions on an abelian surface of Picard rank 1. For the original construction, we refer the reader to \cite{br07},\cite{bridgelandk3}, and \cite{br09}. For a survey, we refer the reader to \cite{bayerli} and \cite{bayer18}. Much of this section is adapted from the background sections of these two papers.\par
Let $X$ be an abelian surface of Picard rank 1, and let $H$ be the ample generator with $H^2=2d$. For $\beta\in \RR$, the $\beta$-twisted Chern character of a coherent sheaf $E$ is defined as
\[
\mathrm{ch}^\beta(E) = e^{-\beta H}.\mathrm{ch}(E),
\]
where $\mathrm{ch}(E)$ denotes the Chern character of $E$. The $\beta$-twisted slope of $E$ is defined as
\[
\mu_{H,\beta}(E)=\begin{cases}
\frac{\mathrm{ch}_1^\beta(E).H}{\mathrm{ch}_0(E)H^2} & \text{ if $\mathrm{ch}_0(E)> 0$}\\
+\infty & \text{ if $\mathrm{ch}_0(E)= 0$}
\end{cases}
\]

This allows one to define the notion of $\beta$-slope semistability analogous to the usual notion. From this, we can construct the following torsion pair.
\begin{align*}
\mathcal{T}^\beta &:= \langle E \;|\; \text{E is $\mu_{H,\beta}$-semistable with $\mu_{H,\beta}(E)>0$}\rangle.\\
\mathcal{F}^\beta &:= \langle E \;|\; \text{E is $\mu_{H,\beta}$-semistable with $\mu_{H,\beta}(E)\leq 0$}\rangle,
\end{align*}
where $\langle \cdot\rangle$ denotes extension closure. Using this, we can define a heart of a bounded t-structure in $\mathrm{D}^b(X)$:
\[
\mathrm{Coh}^\beta(X) := \langle \mathcal{F}^\beta[1],\mathcal{T}^\beta\rangle := \{ E \; |\; H^{-1}(E)\in \mathcal{F}^\beta, H^0(E)\in \mathcal{T}^\beta, H^i(E)=0 \text{ for $i\neq 0,-1$}\}.
\]
Now let $t>0$ and keep $\beta$ as above. We define the \emph{central charge} $Z_{\beta,t}:K(X)\to \CC$ by
\[
Z_{\beta,t}(E)=\langle e^{(\beta+it)H},\mathbf{v}(E)\rangle,
\]
where $\langle\cdot \rangle$ here is the Mukai pairing. 

\begin{defn}
For $t>0,\beta\in \RR$, the pair $\sigma_{(\beta,t)}=(\mathrm{Coh}^\beta(X),Z_{\beta,t}(E))$ is a \emph{Bridgeland stability condition}.
\end{defn}

The significance of this is that, given a stability condition $\sigma_{(\beta,t)}=(\mathrm{Coh}^\beta(X),Z_{\beta,t}(E))$, we can define the following slope function for objects $E\in \mathrm{Coh}^\beta(X)$:
\[
\mu_{(\beta,t)}(E)= \begin{cases}
\frac{-\mathrm{Re}Z_{\beta,t}(E)}{\mathrm{Im}Z_{\beta,t}(E))} & \text{ if $\mathrm{Im}Z_{\beta,t}(E))>0$}\\
+\infty & \text{ if $\mathrm{Im}Z_{\beta,t}(E))=0$}
\end{cases}
\]
This then leads to a familiar notion of slope-stability in $\mathrm{Coh}^\beta(X)$. Unstable objects admit a unique Harder-Narsimhan filtration with semistable factors, and semistable objects admit a Jordan-H\"older filtration with stable factors of the same slope. \par
We note that, for objects $E\in \mathrm{Coh}^\beta(X)$, $Z_{(\beta,t)}$ maps $\mathrm{Coh}^\beta(X)$ to the upper-half plane in $\CC$. Thus, for a stability condition $\sigma_{(\beta,t)}$, we can define the \emph{phase} of $E\in \mathrm{Coh}^\beta(X)$ as
\[
\phi_{(\beta,t)}(E)=\frac{1}{\pi}\mathrm{Arg}(Z_{(\beta,t)}(E))
\]
where $\mathrm{Arg}$ denotes the argument of the complex number $Z_{(\beta,t)}(E)$. \par
For a Bridgeland stabilty condition $\sigma$ on $X$ and a Mukai vector $\mathbf{v}$, let $\mathcal{M}_{\sigma}(\mathbf{v})$ denote the moduli stack of $\sigma$-semistable objects of $\mathrm{Coh}^\beta(X)$, and let $\mathcal{M}_{\sigma}^s(\mathbf{v})$ denote the substack of stable objects. It is a result of Toda in \cite{toda} that $\mathcal{M}_{\sigma}(\mathbf{v})$ is an algebraic stack. There is also a coarse moduli space $M_{\sigma}^s(\mathbf{v})$ of stable object which exists as a smooth, quasi-projective variety. \par
For a fixed Mukai vector $\mathbf{v}$, there is also a wall-and-chamber decomposition of the upper-half plane $\mathcal{H}$ parametrizing stability conditions. There is a \emph{locally finite} set of a walls (real codimension 1 submanifolds) in $\mathcal{H}$ such that
\begin{enumerate}
\item When $\sigma$ varies in a chamber (a connected component of the complement of the union of walls), the set of $\sigma$-stable and $\sigma$-semistable objects do not change. 
\item When $\sigma$ lies on a wall, there is a $\sigma$-semistable object in $E$ that is unstable in one of the adjacent chambers and semistable in the other. 
\item For $\beta<\mu(\mathbf{v})$ and $t \gg 0$, an object of $\mathrm{Coh}^\beta(X)$ is $\sigma_{(\beta,t)}$ semistable if and only if it is a Gieseker semistable sheaf. In other words, for $\sigma_{(\beta,t)}$, the Bridgeland moduli space recovers the Gieseker moduli space.
\end{enumerate}

More results about these Bridgeland moduli spaces will be stated as needed.

\subsection{Fourier-Mukai transforms}
Here, we introduce the necessary background on Fourier-Mukai transforms on abelian surfaces, first introduced by Mukai in \cite{mukai81}. We base this section on the Fourier-Mukai section of \cite{coskunnueryoshioka2}.\par
Let $X$ be a Picard rank 1 abelian surface with ample generator $H$, and let $\hat{X}$ be the dual abelian surface. Let $\mathcal{P}^{\vee}$ denote the dual Poincare line bundle on $X\times \hat{X}$, and let $p,\hat{p}$ be the two projections from $X\times \hat{X}$. The \emph{Fourier-Mukai transform} $\Phi:\mathrm{D}^b(X)\to \mathrm{D}^b(X)$ is defined as
\[
\Phi(E)=\mathbf{R}\hat{p}_*(\mathcal{P}^{\vee}\otimes p^*(E))
\]
The Fourier-Mukai transform induces an isometry of Hodge structure
\begin{align*}
H^{2*}(X,\ZZ)&\to H^{2*}(\hat{X},\ZZ)\\
(r,\xi,a)&\mapsto (a,-\hat{\xi},r)
\end{align*}
where $\hat{\xi}$ is the Poincaré dual of $\xi$. In our applications, we will mainly deal with the shifted functor $\Phi[1]$. In particular, we note that if $E$ has Mukai vector $\mathbf{v}(E)=(r,mH,\chi)$, then 
\[
\mathbf{v}(\Phi[1](E))=(-\chi,m\hat{H},-r),
\]
where $\hat{H}$ is the ample generator on $\hat{X}$. We will need the following facts about $\Phi$.

\begin{theorem}\label{fouriermukai}
Let $X$ and $\Phi$ be as above.
\begin{enumerate}
\item (\cite{mukai81}, Theorem 2.2) The Fourier-Mukai transform $\Phi$ is an equivalence of categories between $\mathrm{D}^b(X)$ and $\mathrm{D}^b(\hat{X})$.
\item (\cite{mukai81}, Proposition 2.7) Let $E\in \mathrm{Coh}(X)$. Then for $x\in \hat{X}$,
\[
H^j(X,E\otimes \mathcal{P}_t)=\Ext_{\hat{X}}^{j+2-i}(\Oc_{-x},\Phi[i](E)).
\]
\item (\cite{yy14}, Theorem 3.8) 
$\Phi[1]$ induces an isomorphism
\[
\mathcal{M}_{(0,t)}(r,mH,\chi)\cong\mathcal{M}_{(0,\frac{1}{dt})}(-\chi,m\hat{H},-r)
\]
where $H^2=2d$.
\end{enumerate}
\end{theorem}

\subsection{Elementary modifications}\label{elemmodprelim}
In this section, we recall some facts about elementary modifications. We base much of this section on Section 5 of \cite{ch16}.

\begin{defn}
Let $E$ be a coherent sheaf and let $E\to \Oc_p$ be a surjective homomorphism. Then the kernel
\[
0\to E'\to E\to \Oc_p\to 0
\]
is called an \emph{elementary modification} of $E$.
\end{defn}
If $E$ has positive rank, we note the following relations:
\begin{enumerate}
\item $\rank(E')=\rank(E)$
\item $c_1(E')=c_1(E)$
\item $\chi(E')=\chi(E)-1$.
\end{enumerate}

We need the following lemma, whose proof is immediate.
\begin{lemma}
If $E$ is $\mu$-(semi)stable, then $E'$ is $\mu$-(semi)stable.
\end{lemma}

\begin{remark}
Elementary modifications do not in general preserve Gieseker (semi)stability. Thus, we need the existence of $\mu$-stable sheaves in our proof.
\end{remark}

Finally, given a short exact sequence of sheaves, there is a natural induced short exact sequence involving elementary modifications. 

\begin{lemma}
Let 
\[
0\to F\to E\to G\to 0
\]
be a short exact sequence of coherent sheaves. Let $G'$ be the elementary modification of $G$ associated to a surjective map $G\to \Oc_p$, and let $E'$ be the elementary modification associated to the composition $E\to G\to \Oc_p$. Then there is a natural short exact sequence 
\[
0\to F\to E'\to G'\to 0,
\]
called the \emph{elementary modification} of the original sequence.
\end{lemma}
\begin{proof}
This follows from the following Snake lemma diagram:
\[\begin{tikzcd}
	&& 0 & 0 & \\
	0 & F & {E'} & {G'} & 0 \\
	0 & F & E & G & 0 \\
	&& {\Oc_p} & {\Oc_p} \\
	&& 0 & 0
	\arrow[from=1-3, to=2-3]
	\arrow[from=1-4, to=2-4]
	\arrow[from=2-1, to=2-2]
	\arrow[from=2-2, to=2-3]
	\arrow[equals, from=2-2, to=3-2]
	\arrow[from=2-3, to=2-4]
	\arrow[from=2-3, to=3-3]
	\arrow[from=2-4, to=2-5]
	\arrow[from=2-4, to=3-4]
	\arrow[from=3-1, to=3-2]
	\arrow[from=3-2, to=3-3]
	\arrow[from=3-3, to=3-4]
	\arrow[from=3-3, to=4-3]
	\arrow[from=3-4, to=3-5]
	\arrow[from=3-4, to=4-4]
	\arrow[equals, from=4-3, to=4-4]
	\arrow[from=4-3, to=5-3]
	\arrow[from=4-4, to=5-4]
\end{tikzcd}\]

\end{proof}

\section{Proof of Nonemptiness}\label{nonempty}
In this section, we prove the following theorem.

\begin{theorem}\label{nonemptiness}
Let $(X,H)$ be a polarized abelian surface of Picard rank 1, with $H$ the ample generator. Let $\mathbf{v}=(r,mH,\chi)$ be the Mukai vector with $r,m\geq 1$ and $\chi\leq 0$. Then the Brill-Noether locus is nonempty and contains $\mu$-stable sheaves.
\end{theorem}
\begin{proof}
We start with the case $\chi=0$. Pick some $E\in M_X(r,mH,0)$ that is $\mu$-stable (see Remark \ref{stablesheavesexist}). We claim that there exists a degree 0 line bundle $L$ such that $H^0(E\otimes L)\neq 0$. For this, apply the Fourier-Mukai transform $\Phi$. Since $\Phi$ is an equivalence, $\Phi(E)\neq 0$, which means $H^i(E\otimes L)\neq 0$ for some $i$ and $L\in \hat{X}$. Note that $H^2(E\otimes L)=0$ for all $L$ by Serre duality and stability. Thus, there exists $L$ such that either $H^0(E\otimes L)$ or $H^1(E\otimes L)\neq 0$. If it is the former, we're done, so suppose $H^1(E\otimes L)\neq 0$. 
But $H^2(E\otimes L)=0$ and $\chi(E\otimes L)= 0$ implies that 
\[
H^0(E\otimes L)\neq 0.
\]
Since $L$ is degree 0, $E\otimes L\in M_X(r,mH,0)$, so $E\otimes L$ is a $\mu$-stable sheaf in the Brill-Noether locus. \par
To extend this to arbitrary $\chi\leq 0$, start with a Mukai vector $\mathbf{v}=(r,mH,\chi)$ satisfying the assumptions of the theorem. Let $E$ be a $\mu$-stable sheaf in the Brill-Noether locus of $M_X(r,mH,0)$ and consider the sequence 
\[
0\to \Oc_X\to E\to F\to 0
\]
induced by the section. Taking $-\chi$ elementary modifications of this sequence yields a sequence 
\[
0\to \Oc_X\to E'\to F'\to 0,
\]
where $E'$ is $\mu$-stable of Mukai vector $\mathbf{v}=(r,mH,\chi)$. Thus, $E'$ is in the Brill-Noether locus of $M_X(\mathbf{v})$.
\end{proof}

\section{Proof of Codimension}\label{codim}
In this section, we prove that the Brill-Noether locus has the expected codimension. The proof will be an induction on $m$. For the induction to work, we actually need to prove a stronger statement about the Bridgeland moduli spaces.

\begin{defn}
Let $\sigma_t=\sigma_{(0,t)}$ be a stability condition. The \emph{Brill-Noether locus} in $\mathcal{M}_{\sigma_t}(\mathbf{v})$ is the locus of all $E\in \mathcal{M}_{\sigma_t}(\mathbf{v})$ where $\Hom(\Oc_X,E)\neq0$ in the derived category.
\end{defn}

\begin{theorem}\label{BNmain}
Let $(X,H)$ be a polarized abelian surface of Picard rank 1, with $H$ the ample generator. Let $\mathbf{v}=(r,mH,\chi)$ be a Mukai vector with $r,m\geq 1$ and $\chi\leq 0$. Let $\sigma_t=\sigma_{(0,t)}$ be a stability condition not on a wall. Then the Brill-Noether locus in $\mathcal{M}_{\sigma_t}(\mathbf{v})$ has codimension at least $1-\chi$. \par
In particular, since it is known that $\mathcal{B}^1_X(\mathbf{v})\subset \mathcal{M}_X(\mathbf{v})$ has codimension $\leq 1-\chi$, applying the theorem to $t\gg0$ implies that $\mathcal{B}^1_X(\mathbf{v})\subset \mathcal{M}_X(\mathbf{v})$ has the expected codimension of $1-\chi$.
\end{theorem}

\begin{strategy*}
We first describe the strategy and main difficulty of the proof. The basic idea is as follows.
\begin{enumerate}
\item Start from some $\mathcal{M}_{\sigma}(\mathbf{v})$, where $\sigma$ is not on a wall. For $\sigma$ large, this is isomorphic to the Gieseker moduli space $M_{X}(\mathbf{v})$. 
\item Travel along the path $(0,tH)$ in the stability upper half plane until we reach $t\ll1$. Along the way, show that, at each wall, we do not lose too many stable objects in the Brill-Noether locus. Specifically, the codimension of 
\[
\{\text{destabilized objects}\}\cap \{\text{Brill-Noether locus}\}
\]
 should be at least $1-\chi$. 
\item Once we're at $M_{(0,tH)}(\mathbf{v})$ for $t\ll 1$, the Fourier-Mukai transform $\Phi[1]$ gives an isomorphism 
 \[
 M_{(0,tH)}(r,mH,\chi)\cong M_{(0,\frac{1}{dt}H)}(-\chi,m\hat{H},-r),
 \]
 where $\hat{H}$ is the ample generator on the dual abelian variety $\hat{X}$. Since $t\ll 1$, the space on the right is isomorphic to the Gieseker moduli space $M_{\hat{X},\hat{H}}(-\chi,m\hat{H},-r)$.
 
 \item Using that the Fourier-Mukai transform is an equivalence of categories (in particular, it is a fully faithful functor), the above isomorphism identifies the Brill-Noether locus in $M_{(0,tH)}(r,mH,\chi)$ with the locus of sheaves in $M_{\hat{X},\hat{H}}(-\chi,m\hat{H},-r)$ that are singular at the point $[\Oc_X]\in \hat{X}$. We then directly prove that this locus has codimension $1-\chi$. 
 \item By step 2 and local finiteness of walls, we lost at most codimension $1-\chi$ worth of stable objects traveling from $M_{\sigma}(\mathbf{v})$ to $M_{(0,tH)}(\mathbf{v})$, so this proves the theorem. 
\end{enumerate}

Most steps of this argument follow readily from the foundational theory of Bridgeland stability conditions and Fourier-Mukai transforms on abelian surfaces. The key computation is showing that, at each wall, we do not lose too many stable objects. This can be done by looking at the Harder-Narasimhan filtrations of the destabilized objects. \par
The obvious thing to do would be to crudely bound the locus of \emph{all} destabilized objects and hope that the bound we get suffices. The difficulty is that, as it turns out, the locus of all destabilized objects sometimes has codimension $<1-\chi$ (see Example \ref{BNexample}). In these cases, we really need to bound the locus of stable objects \emph{in the Brill-Noether locus} that become unstable. 
\end{strategy*}

\begin{remark}
To bound the lost objects in the Brill-Noether locus, we will need an induction argument on $c_1$. As mentioned in the introduction, the strategy of bounding objects lost at each wall, as well as the strategy of inducting on $c_1$, is very similar to the methods used in \cite{gl26}.
\end{remark}

Before starting the proof, we first establish our setup. For general preliminaries on Bridgeland stability conditions on abelian surfaces, see Section \ref{prelim}.

\begin{setup*}
We will be concerned with the stability conditions 
\[
\sigma_t=\sigma_{(0,t)}.
\] 
For these stability conditions, the heart $\mathrm{Coh}^\beta(X)=\mathrm{Coh}^0(X)$ is constant, while the central charge is given explicitly by 
\begin{equation}
\label{centralcharge} Z_{\sigma_t}(r,m,\chi)=(-\chi+rdt^2)+(2dmt)i
\end{equation}
For a stability condition $\sigma$, $\mathcal{M}_\sigma(\mathbf{v})$ will denote the moduli stack of $\sigma$-semistable objects of $\mathrm{Coh}^0(X)$.
\end{setup*}

\subsection{Wall crossing and reduction to numerics}
In this subsection, we consider the problem of bounding the locus of stable objects that become unstable after crossing a wall. In particular, we reduce the problem to numerics. 
\par

The key input is the following theorem of Nuer and Yoshioka.

\begin{theorem}[\cite{nueryoshioka}, Theorem 4.2] \label{nueryoshioka}
Let $X$ be an abelian surface and let $\sigma$ be a stability condition. Let $\mathbf{v}$ be a Mukai vector, and $\mathbf{v}_1,...,\mathbf{v}_s$ Mukai vectors with the same phase $\phi$ with respect to $\sigma$, and such that $\sum \mathbf{v}_i=\mathbf{v}$. Let $\sigma_-$ be a stability condition sufficiently close to $\sigma$. Consider the stack 
\[
\mathcal{F}(\mathbf{v}_1,...,\mathbf{v}_s)^*
\]
consisting of Harder-Narasimhan filtrations of objects of $\mathcal{M}_\sigma(\mathbf{v})$, where the $i^{\text{th}}$ semistable factor has Mukai vector $\mathbf{v}_i$. Then $\mathcal{F}(\mathbf{v}_1,...,\mathbf{v}_s)^*$ is an Artin stack with 
\[
\dim \mathcal{F}(\mathbf{v}_1,...,\mathbf{v}_s)^*=\sum \dim \mathcal{M}_{\sigma_-}(\mathbf{v}_i)+\sum_{i<j}\langle \mathbf{v}_i,\mathbf{v}_j\rangle
\]
\end{theorem}

\begin{remark}
The theorem is actually proven in much larger generality, following ideas of \cite{bridgelandmotivic}. The base scheme $X$ can be any smooth projective variety satisfying openness of stability and boundedness of Bridgeland semistable objects as in \cite{toda}, Lemma 3.4. In particular, we just need that the moduli stack of Bridgeland semistable objects be an Artin stack of finite type over $\CC$. This fact for abelian surfaces in proven in section 4 of the same paper (\cite{toda}).\par
For the general dimension statement, one needs to assume that, for any stability condition $\sigma$ and $E,E'$ in the slice $\mathcal{A}_\sigma$ such that $\phi_{\text{min}}(E')>\phi_{\text{max}}(E)$, we have $\Hom(E,E'[k])=0$ for $k\geq 2$. On abelian surfaces, this follows from Serre duality. 
\end{remark}

\begin{remark}
The dimension statement is easy to understand. The idea is to build the HN filtration and count dimensions at each step. So suppose we want to build filtrations
\[
0=\mathcal{E}_0\subset \mathcal{E}_1\subset ...\subset \mathcal{E}_n
\]
where $\mathcal{F}_i=\mathcal{E}_i/\mathcal{E}_{i-1}$ is $\sigma_-$-semistable. At the $i^{\text{th}}$ step we need to construct extensions
\[
0\to \mathcal{E}_{i-1}\to \mathcal{E}_i\to \mathcal{F}_i\to 0
\]
The dimension of the space of $\mathcal{E}_{i-1}$'s has already been counted by the previous steps. The dimension of the space of $\mathcal{F}_i$'s is $\dim \mathcal{M}_{\sigma_-}(\mathbf{v}_i)$. Finally, for a fixed $\mathcal{E}_{i-1}$ and $\mathcal{F}_i$, the space of extensions should be thought of as the quotient stack 
\[
[\Ext^1(\mathcal{F}_i,\mathcal{E}_{i-1})/\Hom(\mathcal{F}_i,\mathcal{E}_{i-1})]
\]
This was originally proven in setting of coherent sheaves by Bridgeland in \cite{bridgelandmotivic} (see also proposition 1.7 of \cite{desaleramanan}). Since this is a quotient stack it has dimension 
\[
\ext^1(\mathcal{F}_i,\mathcal{E}_{i-1})-\hom(\mathcal{F}_i,\mathcal{E}_{i-1}).
\]
With the vanishing condition in the previous remark, this is exactly $-\chi( \mathcal{F}_i,\mathcal{E}_{i-1})$, which is computed as 
\[
\langle \mathbf{v}_i,\sum_{j<i} \mathbf{v}_j\rangle=\sum_{j<i}\langle \mathbf{v}_i,\mathbf{v}_j\rangle.
\]
\end{remark}

We also need the following theorem on dimension. 

\begin{theorem}[\cite{myy2}, Theorem 1.4 and \cite{myy1}, Lemma 5.3.2]\label{modulidim}
Let $(X,H)$ be a polarized abelian surface of Picard rank 1 with $H$ the ample generator. Let $\mathbf{v}$ be a Mukai vector with $\mathbf{v}=\ell\mathbf{v}_0$ for a primitive Mukai vector $\mathbf{v}_0$, and let $\sigma$ be a stability condition not on a wall. Then 
\[
\dim \mathcal{M}_{\sigma}(\mathbf{v})=\begin{cases}
\mathbf{v}^2+1 & \text{ if }\mathbf{v}^2>0\\
\mathbf{v}^2+\ell &\text{ if }\mathbf{v}^2=0
\end{cases}
\]
\end{theorem}
\begin{remark}
Theorem 1.4 in \cite{myy2} exhibits an isomorphism between $\mathcal{M}_{\sigma}(\mathbf{v})$ and some Gieseker moduli space, and Lemma 5.3.2 in \cite{myy1} (among more general things) computes the dimension of the Gieseker moduli space. 
\end{remark}

\begin{notation*}
For the remainder of this section, fix a parameter $t$ so that $\sigma=\sigma_t$ is a stability condition on a wall along our path, and let $\sigma_+,\sigma_-$ be stability conditions on either side of the wall sufficiently close to $\sigma$.
 \end{notation*}
 
 The main result of this subsection is the following, which is essentially Proposition 4.3 of \cite{nueryoshioka}, adapted for the numerics of abelian surfaces. 
  
\begin{proposition}\label{wallcrossing}
Let $(X,H)$ be a polarized abelian surface of Picard rank 1, with $H$ the ample generator. Let $\mathbf{v}=(r,m,\chi)$ be a Mukai vector such that $\mathbf{v}^2>0$, $r,m\geq 1$ and $\chi\leq 0$. Let $\mathbf{v}_1,...,\mathbf{v}_s$ be the $\sigma_-$ Harder-Narasimhan type of an object $E\in \mathcal{M}_{\sigma_+}(\mathbf{v})$ that is $\sigma_-$-unstable. Let 
\[
\delta_i=\begin{cases}
1 & \text{ if }\mathbf{v}_i^2>0\\
\ell_i & \text{ if $\mathbf{v}_i^2=0$ and $\mathbf{v}_i=\ell_i\mathbf{w}_i$ for $\mathbf{w}_i$ primitive}
\end{cases}
\] 
Then the locus of all $\sigma_+$-semistable objects lost at $\sigma$ has codimension at least 
\[
\min_{\mathbf{v}_1,...,\mathbf{v}_s}\left\{ \sum_{i<j}\langle \mathbf{v}_i,\mathbf{v}_j\rangle - \sum_{i=1}^s \delta_i+1\right\},
\]
where the minimum is taken over all $\mathbf{v}_1,...,\mathbf{v}_s$ appearing as a Harder-Narasimhan type of a destabilized object.
\end{proposition}

\begin{remark}
 For $\sigma_+$-semistable objects $E$ that become unstable with respect to $\sigma_-$, only finitely many $\sigma_-$ HN types can occur (see \cite{bridgelandk3}). This justifies taking the max. 
\end{remark}

\begin{proof}
 Since there are only finitely many possible HN types, we can fix an HN type that occurs and bound the locus of objects $E\in \mathcal{M}_{\sigma_+}(\mathbf{v})$ with this particular $\sigma_-$ HN type. So fix a sequence 
\[
\mathbf{v}_1,...,\mathbf{v}_s
\]
 of Mukai vectors, with $\sum \mathbf{v}_i=\mathbf{v}$, that occurs as the $\sigma_-$ HN type of an object $E\in M_{\sigma_+}(\mathbf{v})$. Let 
 \[
 \mathcal{Z}=\{E\in \mathcal{M}_{\sigma_+}(\mathbf{v})\;|\; E \text{ has $\sigma_-$ HN type $(\mathbf{v}_1,...,\mathbf{v}_s)$}\}
 \]
 Our goal is then to bound $\dim \mathcal{Z}$.
 \par  
As in the statement of Theorem \ref{nueryoshioka}, consider the stack 
\[
\mathcal{F}(\mathbf{v}_1,...,\mathbf{v}_s)^*
\]
consisting of $\sigma_-$ HN filtrations of objects of $\mathcal{M}_\sigma(\mathbf{v})$, where the $i^{\text{th}}$ semistable factor has Mukai vector $\mathbf{v}_i$. Restricting to the locus $\mathcal{F}(\mathbf{v}_1,...,\mathbf{v}_s)^\circ\subset\mathcal{F}(\mathbf{v}_1,...,\mathbf{v}_s)^*$ whose $E_s$ term is $\sigma_+$-semistable, the image of the map 
\begin{align*}
\mathcal{F}(\mathbf{v}_1,...,\mathbf{v}_s)^\circ&\to \mathcal{M}_\sigma(\mathbf{v})\\
(0\subset \mathcal{E}_1\subset...\subset \mathcal{E}_s)&\mapsto \mathcal{E}_s
\end{align*}
is exactly $\mathcal{Z}$. By uniqueness of the HN filtration, 
\[
\dim\mathcal{Z}\leq \dim \mathcal{F}(\mathbf{v}_1,...,\mathbf{v}_s)^*.
\]
Now note that, if $\mathcal{F}_i$ is a semistable factor of the $\sigma_-$ HN filtration of $E$, then, since $\sigma_-$ is sufficiently close to $\sigma$, 
\[
\phi_\sigma(\mathbf{v}(\mathcal{F}_i))=\phi_\sigma(\mathbf{v})
\]
(see \cite{bayermacri}, Proposition 5.1). Thus, we can apply Theorem \ref{nueryoshioka} to get
\[
\dim\mathcal{Z}\leq \sum_{i=1}^s \dim \mathcal{M}_{\sigma_-}(\mathbf{v}_i)+\sum_{i<j}\langle \mathbf{v}_i,\mathbf{v}_j\rangle.
\]
By Theorem \ref{modulidim}, $\dim \mathcal{M}_{\sigma_-}(\mathbf{v}_i)=\mathbf{v}_i^2+\delta_i$. Since $\mathbf{v}^2>0$ by assumption, $\dim \mathcal{M}_{\sigma_+}(\mathbf{v})=\mathbf{v}^2+1$. Then using $\sum_{i=1}^s \mathbf{v}_i=\mathbf{v}$, we can directly compute that
\[
\operatorname{codim}\mathcal{Z}\geq \sum_{i<j}\langle \mathbf{v}_i,\mathbf{v}_j\rangle - \sum_{i=1}^s \delta_i+1.
\]

\end{proof}

\subsection{Bounding all lost object}
By Proposition \ref{wallcrossing}, we have reduced the problem of bounding the locus of all lost semistable objects to bounding
\[
\sum_{i<j}\langle \mathbf{v}_i,\mathbf{v}_j\rangle - \sum_{i=1}^s \delta_i+1
\]
for a $\sigma_-$ HN type of a $\sigma_+$-semistable object. In this subsection, we show that, in certain cases, this term is at least $1-\chi$. Thus, for these cases, we don't need a more refined bound where we also consider objects in the Brill-Noether locus.

We start some simple but key observations. 

\begin{lemma}\label{phaseequation}
Let $\mathbf{v}$ and $\mathbf{v}_i$ be as in Theorem \ref{wallcrossing}. Write $\mathbf{v}_i=(r_i,m_iH,\chi_i)$. Suppose the ample $H$ has $H^2=2d$. Then 
\[
m_i\geq 1
\]
and
\[
dt^2(mr_i-rm_i)=m\chi_i-\chi m_i,
\]
where $t$ comes from the stability condition on the wall $\sigma=\sigma_t$. 
\end{lemma}
\begin{proof}
Since $\sigma_-$ is sufficiently close to $\sigma$, by Proposition 5.1 of \cite{bayermacri},
\[
\phi_\sigma(\mathbf{v})=\phi_\sigma(\mathbf{v}_i).
\]
The explicit formula for the central charge $Z_t$ (see Equation \ref{centralcharge}) then yields
\[
(-\chi+rdt^2)+(2dmt)i = \lambda\big(-\chi_i+r_idt^2 + (2dm_it)i\big)
\]
for some positive $\lambda\in \RR$ (recall, the central charge maps to the upper half plane). Comparing imaginary parts yields
\[
\lambda m_i = m.
\]
Since $m$ is positive, we get that $m_i\geq 1$ as well. To deduce the equation in the statement, equate the real parts using $\lambda=\frac{m}{m_i}$ and rearrange. 
\end{proof}
\begin{remark}
This lemma, along with the fact that $\mathbf{w}^2\geq 0$ for semistable Mukai vectors $\mathbf{w}$, will be the basis for essentially all of our computations. 
\end{remark}

\begin{lemma}\label{isotropicpair}
Suppose $\mathbf{v}_i=\langle r_i,m_iH,\chi_i\rangle$ is a primitive, isotropic Mukai vector such that $Z_{\sigma}(\mathbf{v}_i)$ is a positive real multiple of $Z_{\sigma}(\mathbf{v})$. Then 
\[
\langle \mathbf{v},\mathbf{u}\rangle \geq 2-\chi.
\]
\end{lemma}
\begin{proof}
Since $\mathbf{v}_i$ is isotropic, we get
\[
\chi_ir_i=dm_i^2.
\]
Thus, $\chi_i$ and $r_i$ are nonzero and have the same sign. First suppose $\chi_i,r_i>0$. We make a couple observations.
\begin{enumerate}
\item First, by Lemma \ref{phaseequation}, 
\[
dt^2(mr_i-rm_i)=m\chi_i-\chi m_i.
\]
Recall, we assumed $\chi\leq 0$. Thus, $\chi_i,r_i>0$ implies that the right hand side is positive. Looking at the left hand side, we get the inequality 
\[
mr_i>rm_i.
\]
\item Since $\mathbf{v}_i^2$ is isotropic, we have
\[
\chi_i=\frac{dm_i^2}{r_i}.
\]
\end{enumerate}
Combining these, we can compute
\begin{align*}
\langle \mathbf{v},\mathbf{v}_i\rangle &= 2dmm_i-\chi r_i-\chi_ir\\
&=2dmm_i-\chi r_i-\frac{dm_i^2}{r_i}r \\
 &>2dmm_i-\chi r_i-\frac{dmm_i}{r_i}r_i \hspace{0.5cm}\\
&=dmm_i-\chi r_i\\
&\geq 1-\chi.
\end{align*}
Thus, $\langle \mathbf{v},\mathbf{v}_i\rangle \geq 2-\chi$.\par
If $\chi_i,r_i<0$, then $mr_i-rm_i<0$. Thus, the equation in Lemma \ref{phaseequation} yields 
\[
m\chi_i < \chi m_i .
\]
multiplying by $r_i$ (noting that $r_i<0$) and using $dm_i^2=\chi_i r_i$, we get
\[
\chi r_i+1\leq dmm_i.
\]
Applying this, we get 
\begin{align*}
\langle \mathbf{v},\mathbf{v}_i\rangle &= 2dmm_i-\chi r_i-\chi_ir\\
&\geq 2+\chi r_i-\chi_ir\\
&\geq 2-\chi,
\end{align*}
where, in the last step, we use that $\chi_i,r_i<0$.
\end{proof}
\begin{remark}
Since $\sum_{i=1}^s\delta_i$ gets large with isotropic factors, this lemma roughly says that $\sum_{i=1}^s\delta_i$ can't get large without $\sum_{i<j}\langle \mathbf{v}_i,\mathbf{v}_j\rangle$ getting suitably large as well. 
\end{remark}

We also need the following positivity lemma. The proof is computational, so we omit the proof and provide a reference for the interested reader.

\begin{lemma}[\cite{myy1}, Lemma 5.3.4]\label{myypositive}
For $\mathbf{v}_i, \mathbf{v}_j$ as above, if $i\neq j$, then
\[
\langle\mathbf{v}_i,\mathbf{v}_j\rangle \geq 1,\hspace{0.3cm}\text{and}\hspace{0.3cm}\langle\mathbf{v}_i,\mathbf{v}_j\rangle\geq 3 \text{ if $\mathbf{v}_i^2$ and $\mathbf{v}_j^2$ are positive.}
\]
In particular, if $\mathbf{v}_i=\delta_i\mathbf{w}_i$, then 
\[
\langle\mathbf{v}_i,\mathbf{v}_j\rangle\geq \delta_i.
\]
\end{lemma}

\begin{notation}\label{notation}
For the remainder of the proof, we use the following notation.
\begin{enumerate}
\item Let 
\begin{align*}
S:&=\sum_{i<j}\langle \mathbf{v}_i,\mathbf{v}_j\rangle\\
C:&=\sum_{i<j}\langle \mathbf{v}_i,\mathbf{v}_j\rangle - \sum_{i=1}^s \delta_i+1
\end{align*}
\item Let $k$ be the number of vectors $\mathbf{v}_i$ with $\mathbf{v}_i^2>0$. 
\item Decompose $S$ into  
\[
S=S_{00}+S_{0+}+S_{++}
\]
 depending on whether the term $\langle \mathbf{v}_i,\mathbf{v}_j\rangle$ consists of 2 isotropic terms, 1 isotropic term, or no isotropic terms.
 \end{enumerate}
\end{notation}

We now get to the first main result of this subsection.
\begin{proposition}\label{isotropicmain}
In the notation above, if $\mathbf{v}_1,...,\mathbf{v}_s$ contains an isotropic factor, then 
\[
S\geq 1-\chi.
\] 
In particular, in the setting of Proposition \ref{wallcrossing}, the locus of $E\in \mathcal{M}_{\sigma_+}(\mathbf{v})$ whose $\sigma_-$ Harder-Narasimhan filtration contains an isotropic factor has codimension at least $1-\chi$.
\end{proposition} 
\begin{proof}
We want to apply the bound in Lemma \ref{isotropicpair} to all isotropic factors. To this end, observe that 
 \begin{align*}
 \sum_{\mathbf{v}_i^2=0} \langle \mathbf{v},\mathbf{v}_i\rangle=2S_{00}+S_{0+}.
 \end{align*}
Now apply Lemma \ref{isotropicpair} to $\mathbf{w}_i$, where $\mathbf{w}_i$ is the primitive Mukai vector with $\mathbf{v}_i=\delta_i\mathbf{w}_i$. This yields
 \[
 2S_{00}+S_{0+}\geq (2-\chi)(\sum_{\mathbf{v}_i^2=0}\delta_i).
 \]
By Lemma \ref{myypositive}, we also have
\[
S_{0+}\geq k(\sum_{\mathbf{v}_i^2=0}\delta_i) \hspace{0.5cm} \text{and}\hspace{0.5cm}S_{++}\geq 3{k\choose 2},
\]
where $k$ is the number of positive square terms. Putting everything together, we can compute
\begin{align*}
C&=S - \sum_{i=1}^s \delta_i+1\\
&=\frac{2S_{00}+2S_{0+}}{2}+S_{++}- \sum_{i=1}^s \delta_i+1\\
&\geq \frac{2-\chi+k}{2}\left(\sum_{\mathbf{v}_i^2=0}\delta_i\right)+3{k\choose 2}-\sum_{i=1}^s\delta_i+1.
\end{align*}
If $\sum_{\mathbf{v}_i^2=0}\delta_i\geq 2$, it is easy to check that the above expression is $\geq 1-\chi$.\par
If $\sum_{\mathbf{v}_i^2=0}\delta_i=1$, then there is exactly one $\mathbf{v}_i$ which is isotropic, and all $\delta_i=1$. In particular, $S_{00}=0$, which means the first bound above reduces to $S_{0+}\geq 2-\chi$.  Thus, we get
\begin{align*}
C&=S_{0+}+S_{++}-(k+1)+1\\
&\geq (2-\chi)+3{k\choose 2}-k\\
&\geq 1-\chi,
\end{align*}
since $k\geq 1$. 
\end{proof}

So we may now assume all $\mathbf{v}_i^2>0$. The second main result of this subsection is the following. 

\begin{proposition}\label{positive1}
Let $\mathbf{v}_i=\langle r_i,m_iH,\chi_i\rangle$ with $\mathbf{v}_i^2>0$ be an HN factor, and suppose either $r_i<0$ or $\chi_i\geq 0$. Then $C\geq 1-\chi$. 
\end{proposition}
\begin{proof}
The proof is very similar to that of Lemma \ref{isotropicmain}. The idea is that, in both cases, $\langle \mathbf{v},\mathbf{v}_i\rangle -\mathbf{v}_i^2$ is large, so $C$ will be large as well.
\begin{enumerate}
\item First suppose $r_i<0$. Then, by Lemma \ref{phaseequation}, 
\[
m\chi_i<\chi m_i.
\]
Because $\chi\leq 0$ and $m_i\geq 1$, this implies that $\chi_i<0$. Now subtract both sides by $m_i\chi_i$ and multiply both sides by $r_i$ to get
\[
r_i(m-m_i)\chi_i > r_im_i(\chi-\chi_i)
\]
Applying $dm_i^2>r_i\chi_i$ (from $\mathbf{v}_i^2>0$) yields 
\[
dm_i(m-m_i) - r_i(\chi-\chi_i)\geq 1.
\]
Now, as noted in the proof of Lemma \ref{isotropicmain}, every term in $S$ is bounded below by $3$. Thus, isolating the $i^{\text{th}}$ entry, we get
\[
C\geq \langle \mathbf{v},\mathbf{v}_i\rangle -\mathbf{v}_i^2 +3{s-1\choose 2}-s+1.
\]
Since $3{s-1\choose 2}-s+1\geq -1$ (here $s\geq 2$), it suffices to show that $\langle \mathbf{v},\mathbf{v}_i\rangle\geq 2-\chi$. Writing it out explicitly and using the inequality derived above, we get
\begin{align*}
\langle \mathbf{v},\mathbf{v}_i\rangle -\mathbf{v}_i^2&=2dm_i(m-m_i)-(\chi -2\chi_i)r_i -r\chi_i\\
&=2(dm_i(m-m_i)-r_i(\chi-\chi_i))+r_i\chi -r \chi_i r \\
&\geq 2-\chi
\end{align*}

\item Now suppose $\chi_i\geq 0$. Again using Lemma \ref{phaseequation}, this implies that
\[
mr_i>rm_i.
\]
Again, this implies that $r_i>0$. Now, similar to the previous case, we subtract by $r_im_i$, multiply by $\chi_i$, and apply $dm_i^2>\chi_im_i$ to get
\[
dm_i(m-m_i)-\chi_i (r-r_i)\geq 1.
\]
Thus, we get
\begin{align*}
\langle \mathbf{v},\mathbf{v}_i\rangle -\mathbf{v}_i^2&=2dm_i(m-m_i)-\chi r_i +(2r_i-r)\chi_i\\
&>2(dm_i(m-m_i)-\chi_i (r-r_i))+\chi_i r-\chi r_i \\
&\geq 2-\chi
\end{align*}
We now conclude as in the previous case.
\end{enumerate}
\end{proof}

Observe that, if $\chi=0$, having two nontrivial HN factors implies that one of them must have $\chi_i\geq 0$. Thus, this covers all cases for Mukai vectors with $\chi=0$. For the purposes of Section \ref{segresection}, we record this in the following proposition.

\begin{proposition}\label{codim1}
Let $(X,H)$ be a polarized abelian surface of Picard rank 1, with $H$ the ample generator. Let $\mathbf{v}=(r,mH,0)$ be a Mukai vector such that $\mathbf{v}^2>0$ and $r,m\geq 1$. Let $\sigma$ be a stability condition on a wall along the path $\sigma_{(0,tH)}$ and let $\sigma_+,\sigma_-$ be stability conditions sufficiently close to $\sigma$ on either sides of the wall. Then
\[
\operatorname{codim} \{E \in \mathcal{M}_{\sigma_+}(\mathbf{v})\;|\; \text{$E$ is $\sigma_-$-unstable}\}\geq 1.
\]
\end{proposition} 

\begin{remark}
This can actually be proven directly using existing results. Let $\mathbf{v}$ be a Mukai vector. Then for a Mukai vector $\mathbf{v}_1$, consider the following subset in the upper half plane $\mathcal{H}$:
\[
W_{\mathbf{v}_1} = \{(\beta,t)\;|\;\RR Z_{(\beta,t)}(\mathbf{v}_1) =\RR Z_{(\beta,t)}(\mathbf{v})\}.
\]
Every wall in the upper half plane of stability conditions is of this form. Theorem 5.3.5 of \cite{myy1} then states that the stable objects lost at a wall $W_{\mathbf{v}_1}$ has codimension $\geq 1$ if either $\langle \mathbf{v}_1,\mathbf{v}\rangle\neq1$ or $ \mathbf{v}_1^2\neq 0$. One can use Lemma \ref{phaseequation} to show that any $\mathbf{v}_1$ of an actual object satisfying $\langle \mathbf{v}_1,\mathbf{v}\rangle=1$ and $ \mathbf{v}_1^2= 0$ cannot have the same phase as $\mathbf{v}$ along $\sigma_t$. We leave this as an exercise for the reader.
\end{remark}

\subsection{Bounding lost objects in the BN locus}
In the remaining cases, we will need to bound the lost objects \emph{in the Brill-Noether locus}. This will actually require assuming Theorem \ref{BNmain} for smaller $m$, which means this section should be thought of as being part of an inductive argument proving Theorem \ref{BNmain}. We may also assume $\chi\leq -1$. 

\begin{assumptions}\label{assumption}
For the remainder of the subsection, we assume the following.
\begin{enumerate}
\item Assume $C< 1-\chi$ (there is nothing to prove otherwise).
\item Let $\mathbf{v}_i=\langle r_i,m_iH,\chi_i\rangle$
\begin{enumerate}
\item Assume $\mathbf{v}_i^2>0$ for all $i$. 
\item Assume $r_i\geq 0$ and $\chi_i\leq -1$.
\end{enumerate}
\end{enumerate}
Assumptions 2(a) and 2(b) are due to Propositions \ref{isotropicmain} and \ref{positive1}. 
\end{assumptions}

The key observation is the following. 

\begin{lemma}\label{homF}
Suppose $E$ is in the Brill-Noether locus of $M_{\sigma_+}(\mathbf{v})$, meaning $\Hom(\Oc_X,E)\neq 0$ in the derived category. Then $\Hom(\Oc_X,\mathcal{F}_i)\neq 0$ for a $\sigma_-$-HN factor $\mathcal{F}_i$. 
\end{lemma}
\begin{proof}
Each step in the HN filtration gives an exact triangle 
\[
\mathcal{E}_i\to \mathcal{E}_{i+1}\to \mathcal{F}_i\to \mathcal{E}_i[1].
\]
Applying $\Hom(\Oc_X,-)$ yields an exact sequence 
\[
\Hom(\Oc_X,\mathcal{E}_i)\to\Hom(\Oc_X,\mathcal{E}_{i+1})\to \Hom(\Oc_X,\mathcal{F}_i).
\]
Since $\mathcal{E}_1=\mathcal{F}_1$ and $\mathcal{E}_n=E$, if $\Hom(\Oc_X,\mathcal{F}_i)=0$ for all $i$, applying the above exact sequences recursively shows that $\Hom(\Oc_X,E)=0$, a contradiction.
\end{proof}

Recall from Proposition \ref{wallcrossing}, we had first counted the dimension of all lost objects by counting parameters in constructing a HN filtration of type $\mathbf{v}_1,...,\mathbf{v}_s$. The above lemma tells us that, if we only look at the lost objects in the Brill-Noether locus, then one of the HN factors $\mathcal{F}_i$ must be in the Brill-Noether locus of its moduli stack. In particular, since $1\leq m_i<m$, we can use induction on $m$ to bound the dimension of such $\mathcal{F}_i$. The precise statement we need is the following. 

\begin{lemma}\label{wallcrossing2}
Suppose we're in the setting of Proposition \ref{wallcrossing}. Let $\mathcal{Z}$ be the locus of all semistable objects of $M_{\sigma_+}(\mathbf{v})$ lost at $\sigma$ with $\sigma_-$-HN type $\mathbf{v}_1,...,\mathbf{v}_s$, such that the $i^{\text{th}}$ HN factor $\mathcal{F}_i$ has $\Hom(\Oc_X,\mathcal{F})\neq 0$. Suppose the Brill-Noether locus in $\mathcal{M}_{\sigma-}(\mathbf{v}_i)$ has codimension $c_i$. Then $\mathcal{Z}$ has codimension at least 
\[
 \sum_{i<j}\langle \mathbf{v}_i,\mathbf{v}_j\rangle - \sum_{i=1}^s \delta_i+1 + c_i.
\]
In our established notation, this is $C+c_i$.
\end{lemma}
\begin{proof}
This follows from the proof of Proposition 4.2 in \cite{nueryoshioka} and the proof of Proposition \ref{wallcrossing}, but with $\mathcal{M}_{\sigma_-}(\mathbf{v}_i)$ replaced by its Brill-Noether locus. The idea is simply that, at the $i^{\text{th}}$ step, we build extensions
\[
0\to \mathcal{E}_i\to \mathcal{E}_{i+1}\to \mathcal{F}_i\to 0,
\]
where $\mathcal{F}_i$ is taken to be in the Brill-Noether locus of $\mathcal{M}_{\sigma_-}(\mathbf{v}_i)$ instead of $\mathcal{M}_{\sigma_-}(\mathbf{v}_i)$ itself. This doesn't affect the stack of extensions for a fixed $\mathcal{E}_i$ and $\mathcal{F}_i$, since, as we established, this is computed by $\langle \mathbf{v_i},\sum_{j<i} \mathbf{v}_j\rangle$.
\end{proof}

We now prove the main result of this subsection. As explained at the start of the section, the proof will require assuming Theorem \ref{BNmain} for smaller $m$.

\begin{proposition}\label{wallcrossinginduct}
Under assumption (\ref{assumption}), and assuming Theorem \ref{BNmain} for smaller $m$, the codimension of objects in the Brill-Noether locus of $\mathcal{M}_{\sigma_+}(\mathbf{v})$ that become unstable after crossing the wall is at least $2-\chi$.
\end{proposition}
\begin{proof}
First fix a $\mathbf{v}_i$. By Lemma \ref{homF}, it suffices to count the codimension of destabilized objects whose $i^{\text{th}}$ HN factor admits a section. This is enough since there are finitely many $i$. Now, by assumption (\ref{assumption}), we critically have 
\[
\chi_j-\chi>0 \hspace{0.5cm} \text{and} \hspace{0.5cm} r-r_j\geq 0
\]
 for all $j$. First assume the $r_i>0$ (our assumption is only that $r_i\geq 0$). Then the above inequalities yield
 \begin{align*}
\langle \mathbf{v},\mathbf{v}_i\rangle -\mathbf{v}_i^2&=2dm_i(m-m_i)+r_i(\chi_i-\chi)-(r-r_i)\chi_i\\
&\geq 2+\chi_i-\chi.
\end{align*}
Thus, as in the proof of Lemma \ref{positive1},
\[
C\geq \langle \mathbf{v},\mathbf{v}_i\rangle -\mathbf{v}_i^2 -1\geq 1+\chi_i-\chi
\]
Finally, by induction, the Brill-Noether locus of $\mathcal{M}_{\sigma_-}(\mathbf{v}_i)$ has codimension at least $1-\chi_i$. 
By Lemma \ref{wallcrossing2}, the desired codimension is at least
\[
(1+\chi_i-\chi)+(1-\chi_i)=2+\chi.
\]
It remains to deal with the $r_i=0$ case. This will be done in the following lemma, which shows that $r_i$ can't be 0.
\end{proof}
\begin{remark}
In the proof above, we used that the codimension of the Brill-Noether loci in $\mathcal{M}_{\sigma_-}(\mathbf{v}_i)$ had codimension $\geq 1-\chi_i$. This is why we needed to prove the stronger statement about the Brill-Noether locus in $\mathcal{M}_\sigma(\mathbf{v})$ instead of just $\mathcal{M}_X(\mathbf{v})$.
\end{remark}

\begin{lemma}
Let $\mathcal{F}_i\in \mathcal{M}_{\sigma_-}(\mathbf{v}_i)$, where $\mathbf{v}_i=(0,m_iH,\chi_i)$. Then, under assumption (\ref{assumption}), 
\[
\Hom(\Oc_X,\mathcal{F}_i)=0.
\]
\end{lemma}
\begin{proof}
We claim that, for $\mathbf{v}_i=(0,m_iH,\chi_i)$, there are no walls along our path $\sigma_t$. For contradiction, suppose $\sigma=\sigma_t$ is an actual wall, so there is a $\sigma_+$ stable object with some nontrivial $\sigma_-$-HN factor $\mathbf{w}$. First suppose that $\mathbf{w}$ had rank $=0$. Then Lemma \ref{phaseequation} implies that $\mathbf{w}=\lambda \mathbf{v}_i$ for some $\lambda$. But then 
\[
\phi_{\sigma_-}(\mathbf{w})=\phi_{\sigma_-}(\mathbf{v}_i)
\]
contradicting the fact that $\mathbf{w}$ is a nontrivial HN factor of $\mathbf{v}_i$.\par
Thus, there must be some HN factor that has negative rank, say $\mathbf{w}=(a,bH,c)$ with $a<0$. By Lemma \ref{phaseequation}, we have
\[
 1\leq b<m_i \hspace{0.5cm} \text{and}\hspace{0.5cm} dt^2(m_ia)=m_ic-\chi_ib.
\]
Since $a<0$, we have $dt^2(m_ia)<0$ as well, so
\[
m_ic<\chi_i b
\]
Nonnegativity of $\mathbf{w}^2$ yields $\frac{db^2}{a}\leq c$. Thus, we get a chain of inequalities
\[
\chi_i b> m_ic\geq \frac{db^2m_i}{a}\geq -db^2m_i.
\] 
Dividing by $b$ and using $ b<m_i$, we get
\[
\chi_i > -dm_i^2.
\]
\begin{claim*}
This can't happen. In fact, we must have $\chi_i < -dm_i^2$.
\end{claim*}
By Lemma \ref{phaseequation},
\[
m\chi_i<\chi m_i.
\]
Subtracting $m_i\chi_i$ yields
\[
(m-m_i)\chi_i<m_i(\chi-\chi_i).
\]
Thus, $\chi_i \geq -dm_i^2$ would imply that
\[
dm_i(m-m_i)\geq \chi_i-\chi+1.
\]
But then 
\begin{align*}
\langle \mathbf{v},\mathbf{v}_i\rangle-\mathbf{v}_i^2&=2dm_i(m-m_i)-\chi_ir\\
&\geq 2(\chi_i-\chi+1)-\chi_ir\\
&=(\chi_i-\chi)+(1-r)\chi_i+2-\chi\\
&\geq 3-\chi,
\end{align*}
where, in the last step, we used assumption (\ref{assumption}) to get $\chi_i-\chi>0$. Then, as in the proof of Lemma \ref{positive1}, this implies that 
\[
C\geq 1-\chi,
\]
which violates assumption (\ref{assumption}).\par
Thus, we've shown that $\mathbf{v}_i=(0,m_iH,\chi_i)$ has no walls along our path $\sigma_t$. In particular, identifying the Bridgeland moduli space with the Gieseker moduli space, we may take $\mathcal{F}_i$ to be a rank 0 coherent sheaf of pure dimension 1 (since $m_i\neq 0$). Suppose $\Hom(\Oc_X,\mathcal{F}_i)\neq 0$. Such a map would have image $\Oc_D$ for some divisor $D\equiv kH$ for some $1\leq k \leq m_i$. By the ideal sheaf sequence 
\[
0\to \Oc_X(-D)\to \Oc_X\to \Oc_D\to 0,
\]
we get that $\chi(\Oc_D)=-dk^2$. Since $\mathcal{F}_i$ is pure dimension 1, Gieseker stability depends on the Euler characteristic divided by $c_1.H$. Thus, semistability of $\mathcal{F}_i$ implies that
\[
\frac{-dk^2}{2dk}\leq \frac{\chi_i}{2dm_i}.
\]
Rearranging and using $k\leq m_i$ yields
\[
\chi_i\geq-dm_i^2.
\]
But we showed in the claim above that this can't happen.
\end{proof}

This completes all the wall crossing estimates. In summary, we have proven the following.

\begin{proposition}\label{wallcrossingfinal}
Let $(X,H)$ be a polarized abelian surface of Picard rank 1, with $H$ the ample generator. Let $\mathbf{v}=(r,m,\chi)$ be a Mukai vector such that $\mathbf{v}^2>0$, $r,m\geq 1$ and $\chi\leq -1$. Let $\sigma$ be a stability condition on a wall along the path $\sigma_{(0,tH)}$ and let $\sigma_+,\sigma_-$ be stability conditions sufficiently close to $\sigma$ on either sides of the wall. Then, if Theorem \ref{BNmain} is true for smaller $m$, then  
\[
\operatorname{codim}\left( \{E\text{ in the BN locus of } \mathcal{M}_{\sigma_+}(\mathbf{v})\} \cap \{E \text{ is $\sigma_-$-unstable}\} \right)\geq 1-\chi.
\]
\end{proposition}

\begin{example}\label{BNexample}
This example shows that the inductive argument bounding objects in the Brill-Noether loci is necessary, i.e. there is no way to just crudely bound \emph{all} the lost objects. Suppose $H^2=2$ so that $d=1$, and consider 
\[
\mathbf{v}=(2,3H,-11)=(0,H,-4)+(2,2H,-7)=\mathbf{v}_1+\mathbf{v}_2.
\]
One can check that the phases of $\mathbf{v},\mathbf{v}_1$, and $\mathbf{v}_2$ align at $t=\sqrt{\frac{1}{2}}$. Now,
\[
\mathbf{v}_1^2=2,\hspace{0.5cm}\mathbf{v}_2^2=36\hspace{0.5cm}\langle \mathbf{v}_1, \mathbf{v}_2\rangle=12.
\]
Thus, the dimension of the objects with this $\sigma_-$ HN type have codimension 
\[
12-1-1+1=11,
\]
but the codimension we need is $1-(-11)=12$. However, we showed that the $\mathbf{v}_1$ can't have any objects in the Brill-Noether locus, so sections must come from the $\mathbf{v}_2$ factor. The induction hypothesis implies that the Brill-Noether locus in $\mathcal{M}_{\sigma_-}(\mathbf{v}_2)$ has codimension $\geq8$. Thus, the codimension of objects in the Brill-Noether locus with this $\sigma_-$-HN type has codimension at least
\[
11+8=19,
\]
which is greater than the required 12.
\end{example}

\subsection{Completing the argument}
Finally, we put everything together to prove Theorem \ref{BNmain}, which we first recall.

\begin{theorem*}
Let $(X,H)$ be a polarized abelian surface of Picard rank 1, with $H$ the ample generator. Let $\mathbf{v}=(r,mH,\chi)$ be the Mukai vector of a stable sheaf with $r,m\geq 1$ and $\chi\leq 0$. Let $\sigma_t=\sigma_{(0,t)}$ be a stability condition not on a wall. Then the Brill-Noether locus 
\[
\mathcal{B}^1_\sigma(\mathbf{v})\subset \mathcal{M}_\sigma(\mathbf{v})
\]
 has the codimension $\geq 1-\chi$. 
\end{theorem*}
\begin{proof}
We proceed by induction on $m$. For $m=1$, recall that there are no walls along the path $\sigma_t$, since any HN factor $(r_i,m_iH,\chi_i)$ needs to have $m_i\geq 1$ by Lemma \ref{phaseequation}. Thus, the $m=1$ case only deals with the Gieseker space, and this is part of Theorem 1.2 of \cite{bayerli} (see also Theorem 7.9 of \cite{coskunnueryoshioka}). \par
For the inductive step, Proposition \ref{wallcrossingfinal} applies. Thus, at every wall $\sigma$, the locus of semistable objects in the Brill-Noether locus of $\mathcal{M}_{\sigma_+}(\mathbf{v})$ which become $\sigma_-$- unstable has codimension at least $1-\chi$.\par
Now let $\sigma$ be a stability condition not on a wall, and let $0<t\ll1$. By local finiteness of walls and the inductive hypothesis, 
\[
\operatorname{codim}_{\mathcal{M}_\sigma(\mathbf{v})}( \mathcal{B}^1_\sigma(\mathbf{v})\setminus \mathcal{M}_{\sigma_{t}}(\mathbf{v}))\geq 1-\chi,
\]
i.e. the sheaves in $\mathcal{B}^1_\sigma(\mathbf{v})$ that were destabilized before making it to $\mathcal{M}_{\sigma_{t_1}}(\mathbf{v})$ has codimension at least $1-\chi$. Thus, it will suffice to prove that the Brill-Noether locus of $\mathcal{M}_{\sigma_{t}}(\mathbf{v})$ has codimension $1-\chi$.\par
To this end, apply the Fourier-Mukai transform $\Phi[1]$ to $\mathcal{M}_{\sigma_{t}}(\mathbf{v})$. Recall from Theorem \ref{fouriermukai} that $\Phi[1]$ is induces an isomorphism
\[
\Phi[1]:\mathcal{M}_{\sigma_{t}}(\mathbf{v}) \cong \mathcal{M}_{\hat{X}}(-\chi,m\hat{H},-r).
\]
Since $\Phi[1]$ is also an equivalence, if $E\in \mathcal{M}_{\sigma_{t}}(\mathbf{v})$, 
\begin{align*}
\Hom(\Oc_X,E)&\cong \Hom(\Phi[1](\Oc_X),\Phi[1](E))\\
&\cong\Ext^1(\Oc_{0},F),
\end{align*}
where $\Oc_{0}$ denotes the skyscraper sheaf at the point $[\Oc_{\hat{X}}]$ and $F\in \mathcal{M}_X(-\chi,m\hat{H},-r)$ for $\hat{H}$ the ample generator of $\hat{X}$. For $F$, consider the sequence 
\[
0\to F\to F^{**}\to T\to 0.
\]
Applying $\Hom(k_0,-)$ and noting that $F^{**}$ is locally free, we see that 
\[
\Ext^1(k_{0},F)\neq 0 \iff T \text{ is supported at $[\Oc_{\hat{X}}]$}.
\]
 Thus, under $\Phi[1]$, the Brill-Noether locus of $\mathcal{M}_{\sigma_{t}}(\mathbf{v})$ is identified with the locus of sheaves in $\mathcal{M}_X(-\chi,m\hat{H},-r)$ singular at $[\Oc_{\hat{X}}]$. To complete the proof, we will show directly in Theorem \ref{singular} that this locus has codimension $1-\chi$.
 \end{proof}

\begin{theorem}\label{singular}
Let $(X,H)$ be a polarized abelian surface of Picard rank 1 with $H$ the ample generator. Let $\mathbf{v}=(r,mH,\chi)$ be a Mukai vector with $r,m\geq 1$ and $\chi\leq -1$. Let $p\in X$ be a reduced point. Then 
\[
\operatorname{codim} \{E\in \mathcal{M}_X(\mathbf{v}) \;|\; E \text{ is singular at $p$}\} = r+1
\] 
\end{theorem}
\begin{remark}
In remark 3.2 of \cite{yoshioka}, Yoshioka proves this when $r$ and $m$ are coprime. Our proof extends his argument to the arbitrary $m\geq 1$ case. 
\end{remark}
\begin{proof}
As in the case of the Brill-Noether locus, the locus in question also has a determinental structure, so we have the lower bound $r+1$ on the codimension (see Lemma 9.2.1 in \cite{huybrechtslehn}). Thus, it again suffices to prove that the codimension is $\geq r+1$.
Consider the sequence 
\[
0\to E\to E^{**}\to T_E\to 0,
\]
where $T_E$ is a sheaf supported at the singularities of $E$. We record that
\[
\mathbf{v}(E^{**})=(r,mH,\chi+\ell(T_E)).
\]
Note that there are only finitely many possible length of $T_E$. Indeed, if $\Delta$ denotes the discriminant, then
\[
\Delta(E)=\Delta(E^{**})+2r\ell(T_E).
\]
$\Delta(E)$ is a fixed value, and since $E^{**}$ remains $\mu$-semistable, both $\Delta(E),\Delta(E^{**})\geq 0$ by the Bogomolov inequality, thus bounding $\ell(T)$ from above. Thus, it suffices to prove the codimension bound for $E$ with a fixed $\ell(T_E)$. From now on, denote this length by $\ell$. \par
We first observe that, if $F$ is locally free on $X$, then the subscheme of the quot scheme $\operatorname{Quot}(F,\ell)$ consisting of quotients \emph{whose support contains a fixed point $p\in X$} has dimension 
\begin{equation}\label{quotdim}
\ell(r+1)-2.
\end{equation} (note that, since quotients have no automorphisms, the dimension of the stack remains the same). This follows from a quick incidence correspondence argument using Theorem 1 of \cite{ellingsrudlehn}, which computes the dimensions of the quot scheme and punctual quot schemes.

 We consider a few different cases. 
\begin{enumerate}
\item First suppose $\mathbf{v}(E^{**})^2>0$. In this case, the stack of $\mu$-semistable sheaves of Mukai vector $\mathbf{v}(E^{**})^2$ has dimension $\mathbf{v}(E^{**})^2+1$ (by Lemma 3.8 of \cite{kuriharayoshioka}). Thus, applying equation \ref{quotdim} above, the stack of $E$ singular at $p$ with length $\ell$ singular locus has dimension at most
\[
\mathbf{v}(E^{**})^2+1 + \ell(r+1)-2.
\]
Using $\mathbf{v}(E^{**})^2=\mathbf{v}^2-2r\ell$, this simplifies to
\[
\mathbf{v}^2-(r-1)\ell-1.
\]
Thus, the codimension in this stack in $\mathcal{M}_{X}(\mathbf{v})$, which has dimension $\mathbf{v}^2+1$, is at least
\[
(r-1)\ell+2\geq r+1.
\]

\item Now suppose $\mathbf{v}(E^{**})=0$ and $r\geq 2$. Let $\mathbf{v}_i=(r_i,m_iH,\chi_i)$ be the $\mu$-Jordan-H\"older factors of $E^{**}$. Note that 
\[
\frac{m_i}{r_i}=\frac{m}{r}.
\]
Thus, if $H^2=2d$,
\begin{align*}
\sum \frac{\mathbf{v}_i^2}{r_i}&=\sum (\frac{2dm_i^2}{r_i}-2\chi_i)\\
&=\frac{2dm}{r}\sum m_i-\sum 2\chi_i\\
&=\frac{2dm^2}{r}-2\chi(E^{**})\\
&=\frac{\mathbf{v}(E^{**})^2}{r}\\
&=0.
\end{align*}
Since $\mathbf{v}_i^2\geq 0$, we conclude that $\mathbf{v}_i^2=0$ for all $i$. Thus, 
\[
\frac{\chi_i}{r_i}=\frac{\chi(E^{**})}{r}
\]
for all $i$. In particular this shows that all Jordan-H\"older factors of $E^{**}$ have the same reduced Hilbert polynomial as $E^{**}$, so we conclude that $E^{**}$ is Gieseker semistable. Write $\mathbf{v}(E^{**})=k \mathbf{w}$ with $k\geq 1$ for a primitive Mukai vector $\mathbf{w}$. We then have $\dim \mathcal{M}_X(\mathbf{v}(E^{**}))=k$. Thus, as in the previous case, the codimension of the locus in the theorem is at least
\begin{align*}
\mathbf{v}^2+1-(k+(r+1)\ell-2)&=2r\ell-k-(r+1)\ell+3\\
&=(r-1)\ell-k+3,
\end{align*}
where we used that $\mathbf{v}^2=2r\ell$ since $\mathbf{v}(E^{**})^2=0$. Further, $\mathbf{v}(E^{**})^2=0$ also gives us 
\[
\chi(E^{**})=\frac{dm^2}{r}>0.
\]
Since $k$ divides $\chi(E^{**})=\chi+\ell$, we get
\[
k\leq \chi+\ell.
\]
Thus, the codimension is at least
\begin{align*}
(r-1)\ell-k+3&\geq (r-2)\ell-\chi+3\\
&\geq (r-2)-\chi +3\\
&= r+1,
\end{align*}
where we used that $r\geq 2$ implies $(r-2)\ell\geq r-2$.

\item Finally, we deal with the $r=1$ case. But this case is easy since every sheaf is isomorphic to $I_Z\otimes L$ for some $L\in \operatorname{Pic}^m(X)$ and $\ell(Z)=dm^2-\chi$. So being singular at $p$ just means $p\in Z$, which has codimension $2=r+1$ in the relevant Hilbert scheme.\end{enumerate} 
\end{proof}

\section{The Segre Invariant}\label{segresection}

In this section, we will use the results developed in the previous sections to study the rank 1 Segre invariant on Picard rank 1 abelian surfaces. We first recall the definition.

\begin{defn}
Let $(X,H)$ be a polarized smooth projective surface. Let $E$ be a coherent sheaf of rank $r$, and let $1\leq k\leq r-1$. The \emph{$k$th Segre invariant} of $E$ is defined to be
\[
s_k(E)=\min_{F\subset E}\big\{ (kc_1(E)-rc(F)).H\},
\]
where the minimum is taken over all subsheaves $F\subset E$ of rank $k$. \par
Equivalently, the $k$th Segre invariant measures the maximal degree of a rank $k$ subsheaf of $E$.
\end{defn}

It will be helpful to assume that the general sheaf in $M_X(\mathbf{v})$ is locally free and $\mu$-stable. We make a short remark on this condition.

\begin{remark}\label{assumelocallyfree}
Recall from Theorem \ref{sheavesexistence} that, if $r\geq 2$, $\mathbf{v}^2>0$, and $\mathbf{v}\neq (r,0,-1)e^{tH}$ for some $t$, then $M_X(\mathbf{v})$ will contain $\mu$-stable locally free sheaves. \par
To understand this condition, recall that $\mathbf{v}^2>0$ is essentially the condition for the existence of stable sheaves (we need $\mathbf{v}^2\geq 0$ for the existence of semistable sheaves, and if $\mathbf{v}^2=0$, then $M_X(\mathbf{v})$ contains stable sheaves if and only if $\mathbf{v}$ is primitive). Now, for some Mukai vector $\mathbf{v}=(r,mH,\chi)$, if $H^2=2d$, we have
\[
\mathbf{v}^2=2dm^2-2r\chi.
\]
Thus, for a given $\mathrm{rank}=r$ and $\mathrm{ch}_1=mH$, there is a $\chi$ for which $\mathbf{v}=(r,mH,\chi)$ is isotropic if and only if $\frac{dm^2}{r}\in \ZZ$, in which case $\chi=\frac{dm^2}{r}$.\par
Now, observe that $\mathbf{v}=(r,0,-1)e^{tH}$ means
\[
\mathbf{v}=(r,rtH,\frac{d(rt)^2}{r}-1)
\]
and
\[
\mathbf{v}^2=2r.
\]
Thus, to guarantee $\mu$-stable vector bundles in $M_X(\mathbf{v})$, we only need to exclude the following two classes of Mukai vectors:
\begin{enumerate}
\item Isotropic Mukai vectors, and
\item Mukai vectors $\mathbf{v}=(r,mH,\chi)$ where $r\;|\;m$ and $\chi$ is the largest two integers for which $\mathbf{v}^2\geq 0$.
\end{enumerate}
\end{remark}

The following lemma explains why it helps to assume locally free sheaves exist.
\begin{lemma}\label{locallyfree}
Suppose $E$ is a locally free sheaf with $s_1(E)=k$. Then $s_1(E)$ is realized by a subsheaf $L$, where $L$ is a degree $k$ line bundle.
\end{lemma}
\begin{proof}
If $L\subset E$ is a coherent subsheaf of rank 1 that is not a line bundle, then $L^{**}$ is a line bundle with $\deg L^{**}=\deg L$. Since $E$ is locally free, $E$ contains $L^{**}$.
\end{proof}

In view of the above discussion, we will assume our Mukai vectors satisfy the following assumption.

\begin{assumptions}\label{segreassumption} 
For $\mathbf{v}=(r,mH,\chi)$, assume $r\geq 2$, $\mathbf{v}^2>0$, and $\mathbf{v}\neq (r,0,-1)e^{tH}$ for some $t\in \ZZ$. In particular, $M_X(\mathbf{v})$ contains $\mu$-stable locally free sheaves.
\end{assumptions}

\subsection{Generic Segre invariant}
In this subsection, we determine the Segre invariant of the general stable sheaf.

\begin{theorem}\label{genericsegre}
Let $\mathbf{v}=(r,mH,\chi)$ be a Mukai vector satisfying assumption \ref{segreassumption}. Let $k_{\mathbf{v}}$ be the largest integer $k<\mu(\mathbf{v})$ such that $\chi(L,E)\geq -1$, where $\mathbf{v}(E)=\mathbf{v}$ and $L$ is a degree $k$ line bundle. Then for the general sheaf $E$ in $M_X^s(\mathbf{v})$, 
\[
s_1(E)=(m-rk_{\mathbf{v}})H^2.
\] 
\end{theorem}
\begin{proof}
Since the general sheaf is $\mu$-stable under assumption \ref{segreassumption}, the maximal degree $k$ of a rank 1 subsheaf of a general $E\in M_X(\mathbf{v})$ must satisfy $k<\mu(\mathbf{v})$. Now, for an integer $k<\mu(\mathbf{v})$, consider the incidence correspondence
\[\begin{tikzcd}[column sep=tiny,row sep=large]
	& {\mathcal{I}_k=\{(E,L)\;|\;L\in \operatorname{Pic}^k(X), \Hom(L,E)\neq 0\}} & \\
	{M_X^s(\mathbf{v})} && {M_X^s(\mathbf{v}e^{-kH})}
	\arrow["{\pi_1}"', from=1-2, to=2-1]
	\arrow["{\pi_2=E\otimes L^{\vee}}", from=1-2, to=2-3]
\end{tikzcd}\]
Since the general sheaf in $M_X(\mathbf{v})$ is locally free, the generic Segre invariant is then the largest integer $k$ for which $\pi_1$ is surjective.\par
Let's first study $\pi_2$. Clearly, the image of $\pi_2$ is contained in the Brill-Noether locus of $M_X^s(\mathbf{v}e^{-kH})$. We claim that, in fact, $\pi_2$ maps surjectively onto the Brill-Noether locus. Indeed, if $F\in M_X^s(\mathbf{v}e^{-kH})$ has a section, then tensoring the section
\[
\Oc_X\to F
\]
by any degree $k$ line bundle $L$ shows that $(F\otimes L,L)\in \mathcal{I}$. In fact, if a pair $(E,L)$ has $E\otimes L^{\vee}=F$, then of course 
\[
(E,L)=(F\otimes L,L).
\] Thus, we've shown that $\pi_2$ surjects onto the Brill-Noether locus of $M_X^s(\mathbf{v}e^{-kH})$, and all the fibers are isomorphic to $\operatorname{Pic}^k(X)$. Note that $c_1(\mathbf{v}e^{-kH})$ is a positive multiple of $H$ since $k<\mu(\mathbf{v})$. Thus, if $\chi(\mathbf{v}e^{-kH})\leq 0$, then by Theorem \ref{BNmain}, the Brill-Noether locus of $M_X^s(\mathbf{v}e^{-kH})$ has codimension $1-\chi(\mathbf{v}e^{-kH})$. \par
Now observe that, if $\mathbf{v}=(r,mH,\chi)$, then
\[
\chi(\mathbf{v}e^{-kH})=rdk^2-2dmk+\chi.
\]
In particular, the minimum of this quadratic occurs exactly at
\[
k=\frac{m}{r}=\mu(\mathbf{v}).
\]
Thus, we see that, for $k<\mu(\mathbf{v})$, as $k$ increases, $\chi(\mathbf{v}e^{-kH})$ decreases. So suppose $k$ is large enough that $\chi(\mathbf{v}e^{-kH})\leq -2$, in which case the Brill-Noether locus has codimension $\geq 3$. In particular, since $\dim M_X^s(\mathbf{v}e^{-kH})=\dim M_X^s(\mathbf{v})$, we learn that
\[
\dim \mathcal{I}_k \leq \dim M_X(\mathbf{v})-1,
\] 
so $\pi_2$ can't surject.\par
Thus, it now remains to prove that, for $k$ such that $\chi(\mathbf{v}e^{-kH})\geq -1$, $\pi_1$ surjects. We consider three cases.
\begin{enumerate}
\item Suppose $\chi(\mathbf{v}e^{-kH})\geq 1$. In this case, every sheaf in $M_X(\mathbf{v}e^{-kH})$ has a section. So if we start with an $E\in M_X(\mathbf{v})$, pick any $L\in \operatorname{Pic}^k(S)$. Then there exists a map
\[
\Oc_X\to E\otimes L^\vee.
\]
Twisting down by $L$ shows that $(E,L)\in \mathcal{I}_k$. 

\item Now suppose $\chi(\mathbf{v}e^{-kH})= 0$. In this case, the Brill-Noether locus has codimension 1, which means 
\[
\dim \mathcal{I}_k = \dim M_X(\mathbf{v})+1.
\]
To show that $\pi_1$ surjects, it suffices to show that there exists a dimension 1 fiber. Twisting everything by $\Oc(-kH)$, this is equivalent to showing that, for a Mukai vector $\mathbf{w}=(r,mH,0)$ with $r\geq 2,m\geq 1$, there exists a sheaf $E\in M_X^s(\mathbf{w})$ for which 
\[
\{L\in \operatorname{Pic}^0(X) \;|\; \Hom(L,E)\neq 0\}
\]
is 1-dimensional. To this end, we once again consider the moduli stack $\mathcal{M}_X(\mathbf{w})$, identifying it with the Bridgeland moduli space $\mathcal{M}_{\sigma_{t_0}}(\mathbf{w})$ for $t_0\gg0$. We now proceed as in the proof of Theorem \ref{BNmain}.\par
By Proposition \ref{codim1}, there exists a stable sheaf $E\in \mathcal{M}_X(\mathbf{w})$ that remains in $\mathcal{M}_{\sigma_{t_1}}(\mathbf{w})$ for $t_1\ll1$. Apply the Fourier-Mukai transform $\Phi[1]$ to $\mathcal{M}_{\sigma_{t_1}}(\mathbf{w})$. Then, for $L\in \operatorname{Pic}^0(X)$, we have
\[
\Hom(L,E) \cong \Ext^1(\Oc_{-L},\Phi[1](E)),
\]
where $\Phi[1](E)\in \mathcal{M}_{\hat{X}}(0,m\hat{H},-r)$. So now it will be enough to show that, for the general sheaf $F\in \mathcal{M}_{\hat{X}}(0,m\hat{H},-r)$, there is only a curve's worth of $L\in\hat{X}$ for which 
\[
\Ext^1(\Oc_{L},F)\neq 0.
\]
Recall that a sheaf of Mukai vector $(0,m\hat{H},-r)$ is a pure dimension 1 sheaf supported along some curve $C\subset \hat{X}$. 
\begin{claim*}
$\Ext^1(\Oc_{L},F)\neq 0$ if and only if $L\in C$.
\end{claim*}
\begin{proof}
Since $\mathscr{E}xt(\Oc_L,F)$ is only supported at $L$, it has no higher cohomology. Thus, the local-to-global Ext spectral sequence is trivial, so we can identify
\[
\Ext^1(\Oc_{L},F)=H^0(\mathscr{E}xt^1(\Oc_L,F)),
\]
 But then
\[
H^0(\mathscr{E}xt(\Oc_L,F))=\Ext^1(k_L,F_L),
\]
where, $k_L$ and $F_L$ denote the fibers of $\Oc_L$ and $F$ at the point $L$. If $L\notin C$, then $F_L=0$, so $\Ext^1(\Oc_L,F_L)=0$. If $L\in C$, then, because $F$ was pure dimension 1, $\Hom(\Oc_L,F_L)=0$. Using the fact that $\operatorname{depth}_{\Oc_L}(F_L)\leq 1$, we conclude that $\Ext^1(k_L,F_L)=\Ext^1(\Oc_{L},F)\neq 0$.
\end{proof}
In particular, this shows that every sheaf $F$ in $\mathcal{M}_{\hat{X}}(0,m\hat{H},-r)$ satisfies $\Ext^1(\Oc_L,F)\neq 0$ for exactly a 1-dimensional family of $L\in \hat{X}$, completing the proof of this case.

 \item Finally, suppose $\chi(\mathbf{v}e^{-kH})= -1$. In this case, 
\[
\dim \mathcal{I}_k = \dim M_X(\mathbf{v}),
\]
so it will suffice to show that the general fiber of $\pi_1$ is zero-dimensional. Again we can twist by $\Oc(-kH)$ to reduce this to looking at degree 0 line bundles mapping to sheaves of Mukai vector $\mathbf{w}=(r,mH,-1)$. The same wall-crossing argument used in the proof of Theorem \ref{BNmain} and case (2) above reduces this to showing that, for a general sheaf $F\in \mathcal{M}_{\hat{X}}(1,m\hat{H},-r)$, there are only finitely many $L\in \hat{X}$ for which 
\[
\Ext^1(\Oc_{L},F)\neq 0.
\]
As we argued in the proof of Theorem \ref{BNmain}, this is equivalent to $F$ having finitely many singularities. But this follows directly from $F$ being torsion-free on a surface.
\end{enumerate}

\end{proof}

\subsection{Nonemptiness of Segre strata}
In this subsection, we show that, for most $\mathbf{v}=(r,mH,\chi)$, for every integer $k_{\mathbf{v}}\leq k\leq \lceil \mu(\mathbf{v})\rceil -1$, $(m-rk)H^2$ is the Segre invariant of a sheaf in $M_X(\mathbf{v})$. 

We make the following definition.

\begin{defn}
Let $\mathbf{v}$ be a $M_X^s(\mathbf{v})$ be a Mukai vector satisfying assumption \ref{segreassumption}. For an integer $k$, define the strata
\[
S_k(\mathbf{v})=\{E\in M_X^s(\mathbf{v}) \;|\; s_1(E)=k\}.
\]
\end{defn}

\begin{remark}
Following the proof of Theorem \ref{genericsegre}, the natural argument for proving nonemptiness of each $S_k$ is the following: For $k\geq k_{\mathbf{v}}$, consider the same incidence correspondence
\[\begin{tikzcd}[column sep=tiny,row sep=large]
	& {\mathcal{I}_k=\{(E,L)\;|\;L\in \operatorname{Pic}^k(X), \Hom(L,E)\neq 0\}} & \\
	{M_X^s(\mathbf{v})} && {M_X^s(\mathbf{v}e^{-kH})}
	\arrow["{\pi_1^k}"', from=1-2, to=2-1]
	\arrow["{\pi_2=E\otimes L^{\vee}}", from=1-2, to=2-3]
\end{tikzcd}\]
Then 
\[
\im \pi_1^k\supset \im \pi_1^{k+1},
\]
so we would like to say that, if $\im \pi_1^k\neq \im \pi_1^{k+1}$, then $\im \pi_1^k$ contains a point of $S_k(\mathbf{v})$, i.e. if there exists $E$ with $\Hom(L,E)\neq 0$ for a degree $k$ line bundle, but $\Hom(L',E)=0$ for all degree $k+1$ line bundles. And $\im \pi_1^k\neq \im \pi_1^{k+1}$ would follow from computing 
\[
\operatorname{codim} \im \pi_1^k=-1-\chi(\mathbf{v}e^{-kH}),
\]
and showing that 
\[
\dim \im \pi_1^k > \dim  \im \pi_1^{k+1}.
\] 
However, once $k>k_{\mathbf{v}}$, we don't know that the image of $\pi_1^k$ contains locally free sheaves anymore, which means we don't know that the Segre invariant is realized by a line bundle. In other words, even if $\Hom(L',E)=0$ for all degree $k+1$ line bundles $L'$, $E$ could still admit a singular rank 1 subsheaf of degree $k+1$. Thus, to show nonemptiness, we proceed using a different method.
\end{remark}

The key observation is the following.
\begin{lemma}\label{nonemptylemma}
Let $\mathbf{v}=(r,mH,\chi)$ and $\mathbf{v}_1=(r,mH,\chi-1)$ be adjacent Mukai vectors satisfying assumption \ref{segreassumption}. Then 
\[
k_{\mathbf{v}}-k_{\mathbf{v}_1}= 0\text{ or }1
\]
In other words, decreasing $\chi$ by 1 changes the Segre invariant of the general sheaf by at most 1 step. 
\end{lemma}
\begin{proof}
By Theorem \ref{genericsegre}, $k_{\mathbf{v}}$ is the largest integer $k$ for which $\chi(\mathbf{v}e^{-kH})+1 \geq 0$. Recall that 
\begin{align*}
\chi(\mathbf{v}e^{-kH})+1=rdk^2-2dmk+\chi+1.
\end{align*}
First observe that the discriminant of this quadratic is 
\[
4(dm)^2-4(rd)(1+\chi)=4(d^2m^2-rd(1+\chi))
\]
Suppose that $d^2m^2\leq rd(1+\chi)$, so that the quadratic is nonnegative. Rearranging yields
\[
\frac{dm^2}{r} \leq \chi+1.
\]
But the condition $\mathbf{v}^2>0$ also says that
\[
\frac{dm^2}{r}>\chi,
\]
which means the discriminant is 0 if and only if $\chi$ is maximal among integers satisyfing $\mathbf{v}^2=2drm-2r\chi>0$.
Thus, the generic Segre invariant is computed in two different ways:
\begin{enumerate}
\item If $\mathbf{v}=(r,mH,\chi)$, where $\chi$ is maximal among integers satisfying $\mathbf{v}^2=2drm-2r\chi>0$, then $\chi(\mathbf{v}e^{-kH})+1$ takes nonnegative values. Recall, the minimum of this quadratic occurs at $k=\mu(\mathbf{v})$. Thus, by Theorem \ref{genericsegre},
\[
k_{\mathbf{v}}=\lceil \mu(\mathbf{v})\rceil -1.
\]
(In particular, $(m-r(\lceil \mu(\mathbf{v})\rceil -1))H^2$ is the only integer that occurs as the Segre invariant of a stable sheaf in $M_X(\mathbf{v})$.)
\item Otherwise, $\chi(\mathbf{v}e^{-kH})+1$ has two roots, which means $k_{\mathbf{v}}$ is computed as the largest integer $\leq $ the smaller root. The quadratic formula yields
\[
k_{\mathbf{v}}=\Bigg\lfloor \mu-\frac{\sqrt{d^2m^2-rd(1+\chi)}}{rd}\Bigg\rfloor.
\]
\end{enumerate}
We now consider two cases. First suppose both $\mathbf{v}$ and $\mathbf{v}_1$ do not have maximal $\chi$. In other words, $k_{\mathbf{v}}$ and $k_{\mathbf{v}_1}$ are both computed by method (2) above, so that
\begin{align*}
k_{\mathbf{v}}&=\Bigg\lfloor \mu-\frac{\sqrt{d^2m^2-rd(1+\chi)}}{rd}\Bigg\rfloor\\
k_{\mathbf{v}_1}&=\Bigg\lfloor \mu-\frac{\sqrt{d^2m^2-rd(\chi)}}{rd}\Bigg\rfloor
\end{align*}
To show that $k_{\mathbf{v}}-k_{\mathbf{v}_1}=0$ or $1$, then, it suffices to show that
\[
\frac{\sqrt{d^2m^2-rd(\chi)}}{rd}-\frac{\sqrt{d^2m^2-rd(1+\chi)}}{rd}\leq 1,
\]
which follows from an elementary argument with fractions (noting that $d^2m^2-rd(1+\chi)> 0$ by assumption).\par
The second case is when $\chi$ is maximal so that
\begin{align*}
k_{\mathbf{v}}&=\lceil \mu\rceil-1\\
k_{\mathbf{v}_1}&=\Bigg\lfloor \mu-\frac{\sqrt{d^2m^2-rd(\chi)}}{rd}\Bigg\rfloor.
\end{align*}
In this case, we have to show that
\[
\frac{\sqrt{d^2m^2-rd(\chi)}}{rd}\leq 1.
\]
But here, $\chi$ was assumed to satisfy $d^2m^2-rd(1+\chi)\leq 0$ so, in particular, 
\[
d^2m^2-rd(\chi)\leq rd.
\]
Thus, 
\[
\frac{\sqrt{d^2m^2-rd(\chi)}}{rd}\leq \frac{\sqrt{rd}}{rd}\leq 1,
\]
as desired.
\end{proof}

We now get to the main result of this section.

\begin{theorem}\label{segrenonemptiness}
Let $\mathbf{v}=(r,mH,\chi)$ satisfying assumption \ref{segreassumption}. Let $k_\mathbf{v}$ denote the maximal degree of a rank $1$ subsheaf of the general $E\in M_X(\mathbf{v})$ (i.e. $s_1(E)=(m-rk_{\mathbf{v}})H^2$ for the general $E\in M_X(\mathbf{v})$). Then, for every $k$ for which $k_{\mathbf{v}}\leq k\leq \lceil \mu(\mathbf{v})\rceil-1$, $(m-rk)H^2$ is the Segre invariant of a $\mu$-stable sheaf in $M_X(\mathbf{v})$. 
\end{theorem}
\begin{proof}
For a Mukai vector $\mathbf{v}=(r,mH,\chi)$, let $\mathbf{v}_{\text{max}}=(r,mH,\chi_{\text{max}})$, where $\chi_{\text{max}}$ is the largest integer for which $\mathbf{v}_{\text{max}}=(r,mH,\chi_{\text{max}})$ satisfies $\mathbf{v}_{\text{max}}>0$.
We consider two cases separately.\par 
First suppose $r\nmid m$, so that $M_X(\mathbf{v}_{\text{max}})$ contains $\mu$-stable sheaves (see Remark \ref{locallyfree}). As observed in Lemma \ref{nonemptylemma} above, 
\[
k_{\mathbf{v}_{\text{max}}}=\lceil \mu\rceil -1.
\] 
Now suppose $k$ is between $k_{\mathbf{v}}$ and $\lceil \mu\rceil -1$. By Lemma \ref{nonemptylemma}, there exists some $\mathbf{v}_k=(r,mH,\chi_k)$ for which the Segre invariant of the general sheaf in $M_X(\mathbf{v}_k)$ is $(r-mk)H^2$. Let $E\in M_X(\mathbf{v}_k)$ and consider a sequence 
\[
0\to L \to E\to F\to 0,
\]
where $L$ is a degree $k$ line bundle. By our assumption on $\mathbf{v}_\text{max}$, we may assume $E$ is $\mu$-stable (see Remark \ref{locallyfree}). Thus, applying $\chi_k-\chi$ number of elementary modifications to the sequence above yields a sequence 
\[
0\to L\to E'\to F'\to 0,
\]
where $E'$ is now a $\mu$-stable sheaf in $M_X(\mathbf{v})$ with Segre invariant $(r-mk)H^2$ (since $E'\subset E$, any rank 1 subsheaf of $E'$ with larger degree contradicts the fact that $s_1(E)=(r-mk)H^2$).\par
Now suppose $r\;|\;m$, so that
\[
\mathbf{v}_{\text{max}}=(r,rtH,drt^2-1)
\]
for some $t$. Then $M_X(\mathbf{v}_{\text{max}})$ has no $\mu$-stable sheaves, so the argument above fails since elementary modifications of a Gieseker stable sheaf need not remain Gieseker stable (they only remain slope semistable). Thus, in this case, our starting vector needs to be $\mathbf{v}_{\text{max}'}=(r,rtH,drt^2-2)$. For this Mukai vector, we have
\begin{align*}
k_{\mathbf{v}_{\text{max}}'} &= \Bigg\lfloor \mu-\frac{\sqrt{d^2(rt)^2-rd(drt^2-1)}}{rd}\Bigg\rfloor=\Bigg\lfloor \mu-\frac{\sqrt{rd}}{rd}\Bigg\rfloor=\lceil \mu \rceil-1,
\end{align*}
since $\mu$ is an integer in this case. We then conclude using the same elementary modification argument as above. 
\end{proof}

\bibliographystyle{plain}

\end{document}